\documentclass[11pt, reqno]{amsart}
\usepackage{a4,amsmath,amssymb,amscd,verbatim,bbm,graphicx,enumerate, blkarray}

\usepackage{amsfonts,amsthm}
\usepackage[colorlinks,citecolor=cyan,linkcolor=magenta]{hyperref}
\usepackage{graphicx}
\usepackage{bmpsize}
\usepackage{MnSymbol}
\usepackage{color}
\usepackage{mathrsfs}

\usepackage[nameinlink]{cleveref}

\newtheorem{theorem}{Theorem}
\newtheorem{proposition}[theorem]{Proposition}
\newtheorem{corollary}[theorem]{Corollary}
\newtheorem{lemma}[theorem]{Lemma}

\theoremstyle{definition}

\newtheorem{definition}[theorem]{Definition}

\newtheorem{remark}[theorem]{Remark}

\numberwithin{equation}{section}
\numberwithin{theorem}{section}

\newcommand{\G}{\Gamma}
\renewcommand{\d}{\delta}
\newcommand{\D}{\Delta}

\renewcommand{\k}{\kappa}

\renewcommand{\l}{\lambda}

\renewcommand{\r}{\rho}

\newcommand{\Dc}{{\mathcal D}}

\newcommand{\C}{{\mathbb C}}

\newcommand{\R}{{\mathbb R}}

\newcommand{\Z}{{\mathbb Z}}

\newcommand{\calD}{\mathcal{D}}

\newcommand{\calT}{\mathcal{T}}

\newcommand{\scrD}{\mathscr{D}}

\newcommand{\Hy}{\mathbb{H}}

\newcommand{\al}{\alpha}
\newcommand{\gam}{\gamma}

\newcommand{\del}{\delta}

\newcommand{\Del}{\Delta}
\newcommand{\ep}{\epsilon}
\newcommand{\thet}{\theta}

\newcommand{\lam}{\lambda}

\newcommand\PSL{\operatorname{PSL}}

\newcommand\Aut{\textnormal{Aut}}

\newcommand{\bs}{\backslash}

\newcommand{\ra}{\rightarrow}

\def\({\left(}
\def\){\right)}

\def\l\{{\left\{}
\def\r\}{\right\}}
\def\wt{\widetilde}
\def\wh{\widehat}
\def\gm{\mathfrak{g}}
\def\frakD{\mathfrak{D}}
\def\frakR{\mathfrak{R}}
\def\wbar{\overline}
\def\ov{\overline}
\def\id{{\rm id}}

\def\Cay{{\rm Cay}}
\def\Leb{{\rm Leb}}
\def\ev{{\rm ev\,}}
\def\1{{\bf 1}}

\def\gr{{\rm gr}}
\def\conj{{\bf conj}}

\def\ac{{\rm ac}}
\def\sing{{\rm sing}}

\def\H{{\mathbb H}}

\def\gf{{\frak g}}
\def\bS{{\bf S}}

\def\CAT{{\rm CAT}}

\def\Dil{{\rm Dil}}

\DeclareMathOperator{\diam}{diam}

\newcommand\numberthis{\addtocounter{equation}{1}\tag{\theequation}}

\newtheorem*{remark*}{Remark}

\title[Marked length spectrum rigidity from rigidity on subsets]{Marked length spectrum rigidity from rigidity on subsets}

\author{\small{Stephen Cantrell and Eduardo Reyes}}

\begin{document}
\maketitle
\begin{abstract}
We introduce a new method for studying length spectrum rigidity problems based on a combination of ideas from dynamical systems and geometric group theory. This allows us to compare the marked length spectrum of metrics and distance-like functions coming from various 
geometric origins. Using our new perspective, we provide concise proofs of well-known length spectrum rigidity results and are able to extend classical results to a variety of new settings. Our methods rely on studying Manhattan curves and a coarse geometric analogue of Teichm\"uller space equipped with a symmetrized version of the Thurston metric.
\end{abstract}
\maketitle

\section{Introduction}

Let $S$ be a closed surface of genus at least $2$ and $\gm$ a Riemannian metric on $S$ with negative sectional curvatures. The closed oriented geodesics on $(S,\gm)$ are in bijection with the conjugacy classes (which we denote by $\conj$) in the fundamental group $\G = \pi_1(S)$ of $S$. Let $
\ell_\gm : \conj \to \R_{> 0}
$
be the function that sends a conjugacy class to the length of the corresponding closed geodesic on $(S,\gm)$.
A well-known and celebrated result proved by Otal \cite{otal} and independently by Croke \cite{croke} states that, if $\gm_\ast$ is another negatively curved Riemannian metric on $S$ and $\ell_\gm[x] = \ell_{\gm_\ast}[x]$ for all $[x] \in \conj$, then $(S,\gm)$ and $(S,\gm_\ast)$ are isometric. This result is referred to as \emph{marked length spectrum rigidity}. 

This phenomenon has been extended to more general types of metrics on surfaces such as translation surfaces \cite{dlr}, flat metrics \cite{BL}, and non-positively curved Riemannian metrics with finitely many singularities \cite{con}. Also, Wu recently proved that a large class of surface amalgams are length spectrum rigid \cite{wu}. 
In higher dimensions, marked length spectrum rigidity holds for closed negatively curved Riemannian metrics when one of the metrics is locally symmetric \cite{ham}. A breakthrough result of Guillarmou and Lefeuvre \cite{guilef} asserts that for Riemannian manifolds with Anosov geodesic flow and non-positive curvature, the marked length spectrum locally determines the metric. 

In settings where marked length spectrum rigidity is known or conjectured to hold, there are many other accompanying rigidity questions. For example:
should we need to assume that two metrics have the same marked length spectrum on \textit{all} conjugacy classes to deduce that they are the same or does is suffice to assume equality on a subset of conjugacy classes? If so, which subsets of conjugacy classes determine the full marked length spectrum?

This problem has been considered by many authors and we now outline some known results in the setting of negatively curved surfaces (i.e. $\G = \pi_1(S)$) from above.

We first note that it follows from, although is not explicitly stated in the work of Bonahon \cite{bonahon}, that if $N \trianglelefteq \G$ is an infinite normal subgroup of $\G$ and the marked length spectra of $\gm$ and $\gm_\ast$ agree on conjugacy classes represented by elements of $N$, then they are isometric. Recently, Hao \cite{hao} improved this result (in a more general setting) by showing that if $H < \G$ is a subgroup such that the limit set of $H$ is all of $\partial \widetilde{S}$ where $\widetilde{S}$ is the universal cover of $S$ and $\ell_\gm, \ell_{\gm_\ast}$ agree on $H$, then $\gm,\gm_\ast$ are isometric. Gogolev and Rodriguez Hertz have also recently shown that the length spectrum restricted to closed geodesics lying in a fixed homology class determines the full length spectrum of the metric \cite{gogolev.hertz}. The set $E\subset \conj$ does not have to be induced by a subgroup of $\G$. It is well known (thanks to Fenchel–Nielsen coordinates) that the isometry class of a hyperbolic surface is determined by the lengths of finitely many closed geodesics. For negatively curved Riemannian metrics on closed surfaces, Sawyer \cite{sawyer} 
showed that if the complement $E^c$ of $E \subset \conj$ grows sub-exponentially in the sense that
\begin{equation}\label{eq.seg}
\lim_{T\to\infty}{ \frac{1}{T} \log (\#\{ [x] \in E^c: \ell_{\gm}[x] \leq T\})} = 0
\end{equation}
and $\ell_\gm, \ell_{\gm_\ast}$ agree on $E$, then $\gm, \gm_\ast$ are isometric.  The property of having sub-exponential growth as in (\ref{eq.seg}) is independent of the metric $\gm$. Sawyer's result also follows from an asymptotic formula for the number of conjugacy classes with approximately the same lengths with respect to two metrics $\mathfrak{g}, \mathfrak{g}_\ast$, due to Schwartz and Sharp in constant negative curvature \cite[Thm.~1]{schwartz.sharp} and Dal'bo in variable negative curvature \cite{dalbo}. In both cases, the asymptotic formula is proved using thermodynamic formalism for the geodesic flows. 

Similar marked length spectrum rigidity questions have been considered in other settings. For example, Smillie and Vogtmann \cite{smillie.vogtmann} showed that, in contrast to Teichm\"uller space, there does not exist a finite set of conjugacy classes that can be used to distinguish points in Outer Space. Also, for Kleinian surface groups, Bridgeman and Canary \cite{bridgeman.canary} showed that the collection of  simple closed geodesics on the surface is enough to determine the isometry class of a metric. The question as to whether hyperbolic metrics on a closed surface are determined by their lengths on simple closed geodesics remains open.\\

In this paper we develop a new approach to studying length spectrum rigidity problems such as the ones mentioned above. Our methods combine ideas from geometric group theory and dynamical systems. There are two main benefits to our method. The first is that we can produce concise and simple proofs of well-known rigidity results and secondly, we can generalise these results to many new metrics and geometries.

The problems we consider include:
\begin{itemize}
    \item Subgroup and subset rigidity (i.e. which subsets of conjugacy classes can be used to determine the full length spectrum of a metric as above);
    \item Volume rigidity for Riemannian manifolds (i.e. to what extent does the length spectrum determine the volume of the manifold);
    \item Understanding drops in exponential growth rate for (quasi-convex) subgroups and subsets; and, 
    \item Obtaining uniform estimates for the growth rates of subsets of conjugacy classes that determine the length spectrum (i.e. how large, in terms of a growth rate, is the set of conjugacy classes on which a pair of metrics have the same/close length spectrum).
\end{itemize}

There are many other works considering similar questions related to length spectrum and volume rigidity, see  for example \cite{bou}, \cite{conlaf}, \cite{croke.dairbekov},\cite{gogolev.hertz}, \cite{hao}.

As discussed above, our ideas allow us to extend rigidity results to a variety of new settings and to do so we consider general coarse negatively curved metric spaces. This allows us to compare metrics that are not necessarily Riemannian, and that are defined in terms of different geometric assumptions. Most of the results mentioned above rely on ergodic theoretic machinery associated to a geodesic flow, or at least, on encoding marked length spectra via continuous Busemann cocycles or geodesic currents. As these tools are in general not available outside the Riemannian or low-dimensional setting we need new ideas to tackle these problems. Our proofs are based around using the authors' work on Manhattan curves and a coarse geometric version of Techm\"uller space, $(\scrD_\G, \Delta)$ (see Section \ref{sec.hyp}). This perspective allows us to provide concise proofs of results that were first proved using heavy machinery.

Before stating our results in full generality (in the coarse geometry setting), we present some applications of our general results.

\subsection{Applications of our main theorems}

Our first example is the following theorem which is a homological version of the recent rigidity result of Guillarmou and Lefeuvre  \cite{guilef}.

\begin{theorem}\label{thm.volgl}
    Let $\gm_\ast$ be a smooth, negatively curved Riemannian metric on a closed manifold $M$. For any $N>0$ large enough there exists $\ep>0$ such that for any smooth metric $\gm$ on $M$ satisfying $\|\gm-\gm_\ast\|_{C^N}<\ep$ the following holds: if 
$\ell_{\gm_\ast}[x] \leq \ell_{\gm}[x]$ for every conjugacy class $[x]$ represented by an element in a fixed homology class of $\pi_1(M)$, then 
$$ \textnormal{Vol}(M,\gm_\ast) \leq \textnormal{Vol}(M,\gm).$$
\end{theorem}
This result shows that Guillarmou and Lefeuvre's result holds when we only look at closed geodesics lying in a prescribed homology class. Our methods do not provide a new proof of Guillarmou and Lefeuvre's result but instead allow us to upgrade it. Similarly, our result also generalises a recent result of Gogolev and Rodriguez Hertz \cite[Theorem 6.1]{gogolev.hertz} which in turn generalises a well known result of Croke and Dairbekov \cite[Thm.~1.1]{croke.dairbekov}.

Our proof of this result revolves around studying \textit{Manhattan curves}: curves that encode information about a pair of metrics/geometries. See Section \ref{sec.prelim} for the definition and a discussion on Manhattan curves. 
To illustrate how we use Manhattan curves we briefly describe our approach to proving Theorem \ref{thm.volgl}. 
\begin{enumerate}
    \item First we consider the Manhattan curve for the lifts of $\gm, \gm_\ast$ on $\G=\pi_1(M)$ and, using results from \cite{cantrell.reyes} and \cite{cantrell.tanaka.1}, we show that it is strictly convex (Proposition \ref{thm.stricconvexvsstraight}).
    \item Then using large deviation estimates 
    we deduce rigidity results for $\gm, \gm_\ast$ in terms of growth rates of subsets (Lemma \ref{lem.ineq}).
    \item Lastly we use results from \cite{cantrell.reyes} that shows that the Manhattan curve can be seen as a path that interpolates between metrics in $\scrD_\G$ (a large collection of metrics including $\gm, \gm_\ast$). Applying the results from Step $(2)$ along this path we deduce global bounds on the length spectrum from bounds on homology classes (see the proof of Theorem \ref{thm.domination}).
\end{enumerate}
This method appears to be a unique approach to length spectrum rigidity problems which usually rely on properties of geodesic flows or geodesic currents.

Using our approach we also provide a new proof of the following result.
\begin{theorem}\label{thm.gapword}
    Suppose that $\G$ is a non-elementary hyperbolic group equipped with a finite symmetric generating set $S$ with corresponding word metric $d_S$ and exponential growth rate $v_S$. Then for any quasi-convex subgroup $H < \G$ of infinite index we have that
    \[
    \limsup_{T\to\infty} \frac{1}{T} \log \left( \#\{x \in H : d_S(o,x) < T\} \right)< v_S.
    \]
\end{theorem}

Previous proofs of this result either involve Patterson-Sullivan theory \cite[Cor.~2.8]{MYJ} or counting for regular languages \cite[Thm.~1.1]{dfw}.
Our proof is purely geometric, does not use any boundary theory nor regular languages and relies on the metric properties of $(\scrD_\G, \Delta)$. We in fact prove a more general result, Theorem \ref{thm.growthrateE} from which we can deduce further interesting results concerning exponential drops in growth rate. For example the following result is an immediate corollary of our work.
\begin{theorem}\label{thm.growthtrees}
    Suppose that $\G$ is either a free group or surface group. Let $\G$ admit a small action on the $\R$-tree $\mathcal{T}$. Let $d_\mathcal{T}$ denote the lift of the metric on $\mathcal{T}$ to $\G$, i.e. $d_\mathcal{T}(x,y)$ is the distance in $\calT$ between $x \cdot p$ and $y \cdot p$ where $p \in \mathcal{T}$ is a fixed base point. Then for any word metric $d_S$ on $\G$ and any $L > 0$ we have that
    \[
    \limsup_{T\to\infty} \frac{1}{T} \log \left( \#\{[x] \in \conj: \ell_{\mathcal{T}}[x] < L \text{ and } \ell_S[x] < T \}\right) < v_S.
    \]
    That is, the conjugacy classes of bounded translation length for $\mathcal{T}$ have small exponential growth rate for any word metric on $\G$.
\end{theorem}

We now move on to state the general results from which the above theorems follow and then present further corollaries of our work.


\subsection{The coarse geometric setting}
We now discuss our abstract general results.
Throughout this work, we fix a non-elementary hyperbolic group $\G$, and let $\conj$ denote its set of conjugacy classes. We will work with the collection $\Dc_\G$ of pseudo metrics on $\G$ that are Gromov hyperbolic, $\G$-invariant, and quasi-isometric to a word metric for a finite generating set. In this context, the analogue of metrics being isometric/homothetic is the rough similarity equivalence relation: we say that $d,d_\ast \in \Dc_\G$ are \emph{roughly similar} if there exist $C, \tau >0$ such that $|d(x,y) - \tau d_\ast(x,y)| \le C$ for all $x,y \in \G$. When $d,d_\ast$ are roughly similar with $\tau = 1$ we say that $d,d_\ast$ are \textit{roughly isometric}. 

If $d\in \calD_\G$ and $[x]\in \conj$ is the conjugacy class of $x\in \G$, we consider the quantity
\[
\ell_d[x] := \lim_{n\to\infty} \frac{d(o,x^n)}{n}, 
\]
where $o \in \G$ is the identity element. This number is well-defined for every $[x]\in \conj$, and the function $\ell_d: \conj \ra \R$ is the \emph{(stable) translation length} of $d$. 
For example, if $\gm$ is a negatively curved metric on the closed manifold $M$ with $\G=\pi_1(M)$ 
and $p$ is a point in the universal cover $\widetilde{M}$ of $M$, then $d(x,y)=d^p_{\gm}(x,y):=d_{\tilde\gm}(xp,yp)$ defines a metric on $\calD_\G$, where $d_{\tilde\gm}$ is the length distance on $\widetilde{M}$ induced by the $\G$-invariant lift $\tilde{\gm}$ of $\gm$. In this case we have $\ell_d=\ell_\gm$.

Given $d\in \calD_\G$, the \emph{exponential growth rate} of $E\subset \conj$ with respect to $d$ is the non-negative number 
$$v_d(E):=\limsup_{T\to \infty}{\frac{1}{T} \log (\#\{[x]\in E \colon \ell_d[x]\leq T\})}.$$

The \emph{exponential growth rate} of $d$ is $v_d:=v_d(\conj)$. Note that $v_d$ is positive for every $d\in \calD_\G$, and that it can also be recovered as the limit
$$v_d=\lim_{T\to \infty}{\frac{1}{T} \log (\#\{x\in \G \colon d(o,x)\leq T\})},$$
see Proposition \ref{prop.samegrowthrate}. Marked length spectrum rigidity holds for pseudo metrics belonging to $\calD_\G$ in the following sense: if $d,d_\ast \in \calD_\G$, then they are roughly isometric if and only if $\ell_d[x]=\ell_{d_\ast}[x]$ for every $[x]\in \conj$, see e.g.~\cite[Thm.~1.2]{cantrell.tanaka.1}.

\subsection{Rigidity on subsets}\label{sec.ros}
Now we state our general  results about marked length spectrum rigidity from rigidity on subsets. The first one is as follows. 

\begin{theorem}\label{thm.main}
Let $\G$ be a non-elementary hyperbolic group. There exists a function $\beta: \Dc_\G \times \Dc_\G \to  \mathbb{R}_{\ge 0}$ such that the following holds.
    For any pair $d,d_\ast \in \Dc_\G$ of pseudo metrics with exponential growth rates $v_d, v_{d_\ast} >0$, the following are equivalent:
    \begin{enumerate}
        \item there   exists $x\in\G$ such that $v_d \ell_d[x] \neq v_{d_\ast} \ell_{d_\ast}[x]$; and,
        \item $\beta(d,d_\ast) > 0$.
    \end{enumerate}
    Furthermore, when either (hence both) of these conditions are satisfied we have 
      \begin{equation}\label{eq.ineqbeta}
v_d(E) \le v_d - \beta(d,d_\ast) 
\end{equation}
for any function $f:\R \ra \R_{\geq 0}$ such that $f(t)=o(t)$ as $t \to \infty$ and any subset $E \subset \conj$ for which $|v_d \ell_d[x] - v_{d_\ast} \ell_{d_\ast}[x] |\leq f(\ell_d[x])$ for all $[x]\in E$.
\end{theorem}

From this result, we immediately deduce that if $d,d_\ast\in \calD_\G$ are such that there exists $E \subset \conj$ satisfying:
\begin{enumerate}
\item
$
v_d(E) = v_d; \ \text{and,} 
$

    \item there exists a function $f : \mathbb{R} \to \mathbb{R}_{\geq 0}$ such that $f(t) = o(t)$ as $t\to \infty$ and $|v_d\ell_d[x] - v_{d_\ast}\ell_{d_\ast}[x]| \le f(\ell_d[x])$ for all $[x] \in E$,
\end{enumerate}
then $d$ and $d_\ast$ are roughly similar.

\begin{remark*}
 $1)$  This result applies to metrics with Busemann (quasi-)cocycles that are not necessarily continuous (which is not the case for the other results mentioned above).
 
\noindent $2)$ We have an explicit expression for $\beta(d,d_\ast)$ in terms of the \textit{Manhattan curve} for $d,d_\ast$, see Subsection \ref{sec.manhattan} and Definition \ref{def.beta}.

\noindent $3)$ We also have a version of this result that asks that the pseudo metrics (as opposed to their translation length functions) are close on a subset of group elements, see Theorem \ref{thm.disp}.

\noindent $4)$ Specialized to the case that $d,d_\ast$ are induced by negatively curved Riemannian metrics on a closed surface, this result provides a positive answer to \cite[Conj.~5.4]{sawyer}.

\end{remark*}

It is natural to ask whether $v_d - \beta(d,d_\ast)$ is the optimal upper bound in \eqref{eq.ineqbeta}, i.e. does there exist a set $E$ (satisfying the necessary conditions) such that $v_d(E)=v_d - \beta(d,d_\ast)$. This is the case for metrics coming from Teichm\"uller space (or more generally, from negatively curved surfaces), by Schwartz and Sharp \cite{schwartz.sharp} in constant negative curvature and Dal'bo \cite{dalbo} in variable negative curvature. In Section \ref{sec.auto} we prove that $v_d - \beta(d,d_\ast)$ is attained for pairs of word metrics related by an automorphism, see Theorem \ref{thm.word} and Corollary \ref{coro.betaauto}.

To prove Theorem \ref{thm.main} we use the Manhattan curve for $d,d_\ast$ to explicitly estimate the growth rate $v_d(E)$. The continuity properties of the Manhattan curves allow us to have good control on the quantity $\beta(d,d_\ast)$ as $d$ and $d_\ast$ vary. To make this precise, we recall that the \emph{space of metric structures} on $\G$ is the set $\scrD_\G$ of rough similarity equivalence classes in $\Dc_\G$, and the \emph{symmetric Thurston distance} on $\scrD_\G$ is given by
\begin{equation*}
\Delta([d],[d_\ast]) = \log\left( \Dil(d,d_\ast) \Dil(d_\ast,d) \right), \ \ \text{where} \ \ \Dil(d,d_\ast) = \sup_{[x] \in \conj'}\frac{\ell_d[x]}{ \ell_{d_\ast}[x]}.
\end{equation*}
Here $\conj'\subset \conj$ is the set of non-torsion conjugacy classes and $[d]$ denotes the rough similarity class of $d\in \calD_\G$. The metric space $(\scrD_\G,\Del)$ was studied by the second author in \cite{reyes} and by both authors in \cite{cantrell.reyes}. Continuity of $\beta$ is encoded in the following theorem.
\begin{theorem} \label{thm.continuous}
    Define $\overline{\beta} : \scrD_\G \times \scrD_\G \to \R_{\ge 0}$ by
    \[
\overline{\beta}([d],[d_\ast]) = \frac{\beta(d, d_\ast)}{v_d}
    \]
    where $\beta(d,d_\ast)$ is the function from Theorem \ref{thm.main}.
    Then $\overline{\beta}$ is well-defined and is
    \begin{enumerate}
\item symmetric in its variables;
\item continuous in each variable with respect to the metric $\Delta$; and, 
\item for any $d,d_\ast \in \calD_\G$ we have
    \[
    \overline{\beta}([d],[d_\ast]) \le \frac{e^{\frac{1}{2} \Delta([d],[d_\ast])}-1}{e^{\frac{1}{2} \Delta([d],[d_\ast])}+1}=\tanh \left(\frac{1}{4}\Del([d],[d_\ast])\right).
    \]
    Furthermore, this inequality is strict if  and only if $[d] \neq [d_\ast]$.

\end{enumerate}
\end{theorem}

Note that in the Theorem \ref{thm.main}, the pseudo metrics $v_d d$ and $v_{d_\ast} d_\ast$ have exponential growth rate 1. To obtain marked length spectrum rigidity for pseudo metrics without this assumption, we slightly strengthen our requirements on the subset $E$.
By using a variation of the large deviation result \cite[Thm.~4.23]{cantrell.tanaka.1} we can prove the following.

\begin{theorem}\label{thm.main2}
Let $\G$ be a non-elementary hyperbolic group and take $d,d_\ast \in \Dc_\G$. Suppose that there exists a set $E \subset \conj$ such that
\begin{enumerate}
\item $v_d(E) = v_d$; 
\item $v_{d_\ast}(E)= v_{d_\ast}$; and,
\item there exists a function $f : \mathbb{R} \to \mathbb{R}_{\geq 0}$ such that $f(t) = o(t)$ as $t\to \infty$ and $|\ell_d[x] - \ell_{d_\ast}[x]| \le f(\ell_d[x])$ for all $[x] \in E$.
\end{enumerate}
Then, $d$ and $d_\ast$ are roughly isometric and $\ell_d[x] = \ell_{d_\ast}[x]$ for all $[x] \in \conj$.
\end{theorem}

The same large deviation result allows us to deduce \emph{marked length spectrum domination} on subsets, in the sense of the next theorem.

\begin{theorem}\label{thm.domination}
Let $\G$ be a non-elementary hyperbolic group and consider $d,d_\ast\in \calD_\G$. Suppose that $E \subset \conj$ satisfies:
\begin{enumerate}
\item $v_{\widehat{d}}(E) = v_{\widehat{d}}$ \, for \emph{every} $\widehat{d}\in \calD_\G$; and,
\item there exists a function $f : \mathbb{R} \to \mathbb{R}$ such that $f(t) = o(t)$ as $t\to \infty$ and $\ell_{d_\ast}[x] \leq \ell_{d}[x] +f(\ell_{d}[x])$ for all $[x] \in E$.
\end{enumerate}
Then $\ell_{d_\ast}[x] \leq  \ell_{d}[x]$ for all $[x] \in \conj$.
\end{theorem}

\begin{remark*}
\noindent $1)$ If a set $E \subset \conj$ has complement that grows sub-exponentially as in (\ref{eq.seg}) above, then it satisfies $(1)$ and $(2)$ in Theorem \ref{thm.main2} and $(1)$ in Theorem \ref{thm.domination}. More generally, condition $(1)$ in Theorem \ref{thm.domination} holds if the complement $E^c$ of $E$ in $\conj$ satisfies $v_{\widehat{d}}(E^c)<v_{\widehat{d}}$ for every $\widehat{d}\in \calD_\G$. Complemented with the monotonicity of the area under domination of marked length spectra for negatively curved metrics by Croke and Dairbekov \cite{croke.dairbekov}, this addresses \cite[Conj.~5.4]{sawyer}.

\noindent $2)$ Besides (1) and (2), for Theorem \ref{thm.main2} we do not require any further assumptions on $v_d$ and $v_{d_\ast}$, so a priori we do not impose $v_d=v_{d_\ast}$.

\noindent $3)$ We can replace (1) in Theorem \ref{thm.domination} with the assumption that $v_{\widehat{d}}(E)=v_{\widehat{d}}$ for each $\widehat{d}$ representing a point on the \emph{Manhattan geodesic} for $[d]$ and $[d_\ast]$ (see \cite{cantrell.reyes}). We note that it is not possible to weaken the assumptions and only assume that $E$ satisfies $(2)$ and $v_d(E)=v_d$, $v_{d_\ast}(E)=v_{d_\ast}$ (similarly to Theorem \ref{thm.main2}). This fact follows from the large deviation principle we use in the proof of Theorem \ref{thm.domination}.
\end{remark*}
We now discuss the consequences of these results to the problems mentioned earlier in the introduction. 

\subsection{Subset rigidity}
Using the results mentioned above we can deduce marked length spectrum rigidity from rigidity on subsets of conjugacy classes with properties that do not depend on $d,d_\ast$, but on other geometric or group theoretic assumptions. For example, by applying a result of Coulon, Dougall, Shapira, and Tapie \cite{CDST} we obtain the following corollary. Recall that a (not necessarily normal) subgroup $H < \G$ is called \emph{co-amenable} if the action of $\G$ on the left coset space $H \backslash \G$ admits a $\G$-invariant mean.
\begin{corollary}\label{cor.1}
    Suppose that $\G$ is a non-elementary hyperbolic group and take two pseudo metrics $d,d_\ast \in \Dc_\G$. If there exists a function $f: \mathbb{R} \to \mathbb{R}_{\geq 0}$ such that $f(t) = o(t)$ as $t\to \infty$ and $$d_\ast(o,x) \leq  d(o,x) + f(d(o,x))$$ for all $x$ belonging to a co-amenable subgroup (or a fixed homology class) $H \subset \G$, then $\ell_{d_\ast}[x]\leq \ell_{d}[x]$ for all $[x]\in \conj$. In particular, if $|d(o,x)- d_\ast(o,x)| \leq f(d(o,x))$
    for all $x\in H$, then $d$ and $d_\ast$ are roughly isometric.
\end{corollary}
When $f$ is a constant function, this result is implied in some cases by the work of Hao \cite{hao} in settings where the considered metrics have continuous Busemann cocycles. The corollary above also implies that two negatively curved Riemannian metrics $\mathfrak{g},\mathfrak{g}_\ast$ on a closed surface are isometric if their marked length spectra agree on a fixed homology class. This was recently proved by Gogolev and Rodriguez Hertz \cite{gogolev.hertz} using Livshits theory for abelian cohomology over transitive Anosov flows. Our methods provide a refinement of this result that allows the difference between $\ell_\mathfrak{g}$ and $\ell_{\mathfrak{g}_\ast}$ to be unbounded. In addition, our result applies to arbitrary hyperbolic groups $\G$ and pseudo metrics in $\calD_\G$.





\subsection{Volume rigidity}


Corollary \ref{cor.1} can be used to improve volume estimates induced from marked length spectrum domination. Similar to Theorem \ref{thm.volgl}, we are able to refine a result of Croke and Dairbekov \cite[Thm.~1.1]{croke.dairbekov}.
\begin{corollary}\label{cor.area}
    Suppose that $S$ is a closed surface of genus at least 2, and let $\mathfrak{g}, \mathfrak{g}_\ast$ be two negatively curved Riemannian metrics on $S$. If $\ell_{\mathfrak{g}_\ast}[x] \le \ell_{\mathfrak{g}}[x]$ for each conjugacy class $[x]$ represented by an element in a fixed co-amenable subgroup (resp.~homology class) of $\pi_1(S)$, then 
    \[
    \textnormal{Area}(S,\mathfrak{g}_\ast) \le \textnormal{Area}(S,\mathfrak{g}).
    \]
\end{corollary}

We can also promote volume rigidity for Hitchin representations, where the involved notation is introduced in Subsection \ref{subsec.area}.

\begin{corollary}\label{cor.areahitchin}
    Let $\G$ be the fundamental group of a closed orientable surface of genus at least 2, and let $\rho,\rho_\ast:\R \ra \PSL_m(\R)$ be two Hitchin representations. If these representations satisfy $\lam_1(\rho_\ast(x))\leq \lam_1(\rho(x))$ for all $x$ in a co-amenable subgroup (resp.~homology class), then
    $$\textnormal{vol}_L(\rho_\ast)\leq \textnormal{vol}_L(\rho).$$
\end{corollary}


\subsection{Uniform estimates on growth rates}

We can also use Theorem \ref{thm.continuous} to formulate a local version of Theorem \ref{thm.main} with uniformity over the growth rate constant $\beta(d,d_\ast)$. Given $\del\geq 0$, let $\Dc^\delta_\G \subset \Dc_\G$ denote the collection of pseudo metrics in $\Dc_\G$ that are $\del$-hyperbolic and have exponential growth rate 1.
\begin{corollary}\label{cor.main}
Suppose $\G$ is a non-elementary hyperbolic group and fix $d \in \Dc^\delta_\G$. Then for any  $0<D_1 < D_2$ there exists $\tau = \tau([d],D_1,D_2, \delta) > 0$ such that if $E \subset \conj$ and $d_\ast \in \Dc^\delta_\G$ satisfy
    \begin{enumerate}
        \item $D_1 \le \Delta([d],[d_\ast]) \le D_2$; and,
      \item     $\sup_{[x] \in E}|\ell_{d}[x] - \ell_{d_\ast}[x]|$ is finite,
          \end{enumerate}
   then we have that
           \[
v_d(E) \le  1 - \tau.
    \]
\end{corollary}    
As a corollary we deduce the following result.
Let $\G$ be a hyperbolic surface group and write $\mathcal{QF}_\G$ for the set of quasi-Fuchsian representations of $\G$ into $\PSL_2(\C)$. For each representation $\rho \in \mathcal{QF}_\G$ and $p\in \Hy^3$ we get a metric $d_\rho=d_\rho^p \in \Dc_\G$ which we assume is normalised to have exponential growth rate $1$: this metric is obtained by lifting the metric on $\Hy^3$ to $\G$ via the action of $\rho(\G)$ on the $\rho(\G)$-orbit of $p$. The normalisation constant is the Hausdorff dimension of the limit set, which belongs to $[1,2]$ \cite{bowen}. In this way, by fixing $p\in \Hy^3$ we identify $\mathcal{QF}$ with a subset of $\Dc_\G^\d$ for some $\d=2\log 2$ and Corollary \ref{cor.main} implies the following.

\begin{corollary}
Suppose $\G$ and $\mathcal{QF}_\G$ are as above and fix $p\in \Hy^3$. Then for any $d=d_\rho \in \mathcal{QF}_\G\subset \calD_\G$ and $D_1,D_2 >0$ there exists $\tau > 0$ such that for all $d_\ast \in \mathcal{QF}_\G$ with $D_1 \le \D([d],[d_\ast]) \le D_2$ we have that
\[
v_d(E) \le 1 - \tau
\]
for any set $E$ on which $\sup_{[x]\in E} |\ell_d[x] - \ell_{d_\ast}[x]|$ is bounded.
\end{corollary}
This immediately follows from Corollary \ref{cor.main}.

\subsection*{Organisation}
The organisation of the paper is as follows. In Section \ref{sec.prelim} we cover preliminary material about hyperbolic spaces and and Manhattan curves. In Section \ref{sec.lsos} we prove Theorems \ref{thm.main} and \ref{thm.continuous}. We discuss the optimality of Theorem \ref{thm.main} in Section \ref{sec.auto}, where we show that certain pairs of word metrics attain the maximal growth rate. In Section \ref{sec.mlsdom} we prove Theorems \ref{thm.main2} and \ref{thm.domination}, and in Section \ref{sec.applications} we include some applications of our main theorems about marked length spectrum rigidity, deducing Theorems \ref{thm.volgl}, \ref{thm.gapword} and \ref{thm.growthtrees}, and Corollaries \ref{cor.area}, \ref{cor.areahitchin} and \ref{cor.main}.

\subsection*{Acknowledgements}
We are grateful to Karen Butt for helpful comments and suggestions.


\section{Preliminaries}\label{sec.prelim}
\subsection{Hyperbolic groups, metrics and distance-like functions} \label{sec.hyp}
Let $\G$ be a group.
Recall that a pseudo metric $d$ on $\G$ is called $\del$-\emph{hyperbolic} ($\del\geq 0$) if
\[
(x|y)^d_z \ge \min\{(x|w)^d_z, (y|w)^d_z \} - \delta \text{ for all $x,y,z,w \in \G$}
\]
where
\[
(x|y)^d_z = \frac{1}{2}(d(x,z) + d(z,y) - d(x,y)) \text{ for $x,y,z \in \G$} 
\]
is the Gromov product. The pseudo metric $d$ is \emph{hyperbolic} if it is $\del$-hyperbolic for some $\del$. Throughout this work $\G$ will be a \emph{non-elementary hyperbolic group}: a non-virtually cyclic, finitely generated group $\G$ such that any word metric $d_S$ on $\G$ associated to a finite, symmetric generating set $S$ is hyperbolic. We will write $\Dc_\G$ for the collection of pseudo metrics on $\G$ that are hyperbolic, quasi-isometric to a word metric for a finite generating set and $\G$-invariant.



We say that two pseudo metrics $d,d_\ast$ on $\G$ are \emph{roughly similar} if there exist $\tau, C>0$ such that
\[
|d(x,y) - \tau d_\ast(x,y)| \le C \ \text{ for all $x,y \in \G$}.
\]
If $d$ and $d_\ast$ are roughly similar with $\tau =1$ we say that they are \emph{roughly isometric}.

The \emph{stable translation length} of $d\in \calD_\G$ is given by
\begin{equation}\label{eq.mls}
    \ell_d[x]:=\lim_{n\to \infty}{\frac{1}{n}d(o,x^n)} \hspace{2mm} \text{for } [x]\in \conj,
\end{equation}
which is well-defined by subadditivity. We define the \emph{dilation} of $d,d_\ast \in \Dc_\G$ to be
\[
\Dil(d,d_\ast) = \sup_{[x] \in \conj'}\frac{\ell_d[x]}{ \ell_{d_\ast}[x]}.
\]
Then the \emph{space of metric structures} on $\G$ is the set $\scrD_\G$ of rough similarity equivalence classes in $\Dc_\G$, and the \emph{symmetric Thurston distance} on $\scrD_\G$ is given by
\begin{equation*}
\Delta([d],[d_\ast]) = \log\left( \Dil(d,d_\ast) \Dil(d_\ast,d) \right).
\end{equation*}



\begin{definition}
    Let $\ov\Dc_\G$ be the set of all left-invariant pseudo metrics on $\G$ that have non-constant translation length function and such that there exist $\lambda >0$ and $ d_\ast\in \Dc_\G$ satisfying $(x|y)_o^d \le \lambda(x|y)_o^{d_\ast} + \lambda$ for all $x,y \in \G$. We write $\ov\scrD_\G$ for the quotient of $\ov\Dc_\G$ under the equivalence of rough similarity.
\end{definition}
It was shown in \cite[Lem.~6.3]{cantrell.reyes} that elements in $\ov\Dc_\G$ are hyperbolic pseudo metrics. The stable translation length of a pseudo metric in $\ov\calD_\G$ is defined as in \eqref{eq.mls}. In \cite[Thm.~1.6]{cantrell.reyes} it was shown that the following actions induce points in $\ov\scrD_\G$. 
\begin{enumerate}
    \item Actions on coned-off Cayley graphs for finite, symmetric generating sets, where we cone-off a finite number of quasi-convex subgroups of infinite index.
    \item Non-trivial Bass-Serre tree actions with quasi-convex vertex stabilizers of infinite index.
    \item Non-trivial small actions on $\R$-trees, when $\G$ is a surface group or a free group.
\end{enumerate}
Pseudo metrics corresponding to these actions need not be proper, i.e. given $d \in \ov\Dc_\G$ is possible that there is $T >0$ such that the set $\{x \in \G : d(o,x)  \leq T\}$ is infinite. 

\subsection{Manhattan curves}\label{sec.manhattan}
One of the key tools we utilise are Manhattan curves.
Given a pair of pseudo metrics $d,d_\ast \in \Dc_\G$, its \emph{Manhattan curve} is defined as the boundary of the convex set
\begin{equation} \label{eq.ps}
\mathcal{C}^M_{d_\ast/d} = \left\{ (a,b) \in \mathbb{R}^2 : \sum_{x\in\G} e^{-ad_\ast(o,x) - bd(o,x)} < \infty \right\}.
\end{equation}
We parameterise this curve by a function $\theta_{d_\ast/d} : \R\to\R$: for fixed $a\in\mathbb{R}$, $\theta_{d_\ast/d}(a)$ is the abscissa of convergence of the sum in (\ref{eq.ps}) as $b$ varies. In \cite{cantrell.tanaka.1} the following properties of $\theta_{d_\ast/d}$ were established.
\begin{theorem}\label{thm.manprops}
For $d,d_\ast \in \Dc_\G$ we have the following:
    \begin{enumerate}
    \item $\theta_{d_\ast/d}$ is convex and $C^1$ (i.e. is continuously differentiable);
    \item $\theta_{d_\ast/d}$ goes through the points $(0,v_{d})$ and $(v_{d_\ast},0)$ and is a straight line between these points if and only if $d,d_\ast$ are roughly similar; 
    \item we have that 
    \[
    -\theta'_{d_\ast/d}(v_{d_\ast}) \le  \frac{v_{d}}{v_{d_\ast}} \le -\theta'_{d_\ast/d}(0)
    \] 
    and both equalities occur if and only if $d$ and $d_\ast$ are roughly similar; and,
    \item $d,d_\ast$ are roughly similar if and only if they have proportional marked length spectra.
\end{enumerate}
\end{theorem}
In point $(3)$ we can express 
\[
- \theta'_{d_\ast/d}(0) = \lim_{T\to\infty} \frac{1}{\#\{x \in \G: d(o,x) \leq T\}}\cdot \sum_{x: d(o,x) \leq T} \frac{d_\ast(o,x)}{T}=: \tau(d_\ast/d)
\]
and we call $\tau(d_\ast/d)$ the \emph{intersection number} for $d_\ast,d$. It follows from the definition of the Manhattan curve that $\theta_{d_\ast/d}^{-1} = \theta_{d/d_\ast}$ and so we have that $-\theta'_{d_\ast/d}(v_d) = (-\theta'_{d/d_\ast}(0))^{-1} = \tau(d/d_\ast)^{-1}.$

In certain cases, for example for pairs of word metrics, it is possible to say more about the regularity of $\theta_{d_\ast/d}$ (see \cite{cantrell.tanaka.2}). To prove our theorems we need to improve upon points $(1)$ and $(2)$ above. More precisely, we need to know that $\theta_{d_\ast/d}$ is globally strictly convex if and only if $d,d_\ast$ are not roughly similar. We prove this using ideas from \cite{cantrell.reyes}.

\begin{proposition}\label{thm.stricconvexvsstraight}
    Take pseudo metrics $d,d_\ast \in \Dc_\G$. Then the Manhattan curve $\theta_{d_\ast/d}$ is strictly convex everywhere if and only if $d,d_\ast$ are not roughly similar. In particular, $\theta_{d_\ast,d}$ is either 
    a straight line or strictly convex everywhere.
\end{proposition}
\begin{proof}
We argue by the contrapositive.
Suppose the Manhattan curve $\theta_{d_\ast/d}$ is not strictly convex on its entire domain. Since this curve is $C^1$ and convex, it must contain a line segment, which we suppose is on the interval $[a,b]$ for some $a<b \in \R$. By \cite[Prop.~4.1]{cantrell.reyes} for each $t \in [a,b]$ there is a pseudo metric $d_t \in \Dc_\G$ that is within bounded distance of $td_* + \theta_{d_\ast/d}(t)d$. The Manhattan curve $\theta_{a,b}$ for the pair $d_a, d_b$ is then a straight line on the interval $[0,1]$ and so $d_a$ and $d_b$ are roughly similar by \cite[Thm.~1.1]{cantrell.tanaka.1}. This implies that $d$ and $d_\ast$ are roughly similar and so $\theta_{d_\ast/d}$ is a straight line.
\end{proof}


\section{Marked length spectra on subsets}\label{sec.lsos}
In this section we prove Theorems \ref{thm.main} and \ref{thm.continuous}.
We begin by defining $\beta(d,d_\ast)$, which will be the correlation constant featuring in these results.

\begin{definition}\label{def.beta}
       For $d,d_\ast \in \Dc_\G$  we define
    \[
    \beta(d,d_\ast) = v_d \sup_{t\in [0,1]} (1 - t - \theta_{v_{d_\ast} d_\ast/v_d d}(t)) = \sup_{t\in[0,v_{d_\ast}]}\left( v_d - \frac{v_d}{v_{d_\ast}} t - \theta_{d_\ast/d}(t)\right).
    \]
\end{definition}
The second equality above follows from the definition of the Manhattan curve. We can see $\beta(d,d_\ast)$ as quantifying how convex $\theta_{d_\ast/d}$ is on the interval $(0,v_{d_\ast})$.

\begin{lemma}\label{lem.eqality}
 Given $d,d_\ast \in \Dc_\G$ we have that $\beta(d,d_\ast) = 0$ if and only if $d$ and $d_\ast$ are roughly similar.
\end{lemma}

\begin{proof}
    Since $\theta_{v_{d_\ast}d_\ast/v_d d}$ is convex and goes through $(0,1)$ and $(1,0)$ we see that $\beta(d,d_\ast) = 0$ if and only if $\theta_{v_{d_\ast}d_\ast/v_d d}(t) = 1 -t$ for all $t\in[0,1]$. This is the case if and only if $d$ and $d_\ast$ are roughly similar by Theorem \ref{thm.manprops}.
\end{proof}

It follows from Theorem \ref{thm.manprops} and Proposition \ref{thm.stricconvexvsstraight} that for a pair of non-roughly similar pseudo metrics $d,d_\ast \in \Dc_\G$ with exponential growth rate $1$, there  is a unique solution $\xi \in (0,1)$ to $\theta_{d_\ast/d}'(\xi) = -1$.

\begin{definition}
Given a pair of non-roughly similar pseudo metrics $d,d_\ast \in \Dc_\G$ with exponential growth rate $1$, we define $\xi(d_\ast/d) \in (0, 1)$ to be the unique number satisfying $\theta_{d_\ast/d}'(\xi(d_\ast/d)) = -1$. We also define
\[
\alpha(d_\ast/d) = \xi(d_\ast/d) + \theta_{d_\ast/d}(\xi(d_\ast/d)).
\]
\end{definition}
It turns out that $\alpha(d_\ast/d)$ is symmetric in $d$ and $d_\ast$. The relation between $\alpha$ and $\beta$ is given by the next lemma.
\begin{lemma} \label{lem.sym}
    Suppose that $d,d_\ast \in \Dc_\G$ have exponential growth rate $1$ and are not roughly similar. Then the following identities hold:
    \begin{enumerate}
        \item $\alpha(d_\ast/d) = \xi(d_\ast/d) + \xi(d/d_\ast) = \alpha(d/d_\ast)$; and,
        \item $\beta(d,d_\ast) = 1-\alpha(d_\ast/d).$
    \end{enumerate}
\end{lemma}

\begin{proof}
    (1) Recall that  $\xi(d_\ast/d)$ satisfies $\theta_{d_\ast/d}'(\xi(d_\ast/d)) = -1$. Also, note from the definition of the Manhattan curve that $\theta_{d/d_\ast}$ is the inverse function of $\theta_{d_\ast/d}$. In particular, if we reflect the point $(\xi(d_\ast/d), \theta_{d_\ast/d}(\xi(d_\ast/d)))$ across the line $y=x$ we obtain $(\xi(d/d_\ast), \theta_{d/d_\ast}(\xi(d/d_\ast)))$. This implies that $\theta_{d/d_\ast}(\xi(d/d_\ast)) = \xi(d_\ast/d)$ and $\alpha(d/d_\ast) = \xi(d/d_\ast) + \xi(d_\ast/d)$, as required. 

    \noindent (2) Since $\theta_{d_\ast/d}$ is convex and $C^1$ we see that $f(t) = 1-t - \theta_{d_\ast/d}(t)$ achieves its maximum when $f'(t) = 0$, which is precisely when $t = \xi(d_\ast/d)$. At this point we obtain $f(\xi(d_\ast/d)) = 1 - \alpha(d_\ast/d)$. 
\end{proof}

\begin{remark}\label{rem.symb}
   The lemma above implies that if $d,d_\ast$ have exponential growth rate $1$, then $\beta(d,d_\ast) = \beta(d_\ast, d)$.
\end{remark}

\begin{proof}[Proof of Theorem \ref{thm.main}]
Take two pseudo metrics $d,d_\ast \in \Dc_\G$. Since $v_{\widehat{d}} \widehat{d}$ has exponential growth rate 1 for every $\widehat{d}\in \calD_\G$, without loss of generality we can assume that $v_d=v_{d_\ast}=1$. Under these assumptions, conditions (1) and (2) are equivalent by Lemma \ref{lem.eqality}.

For the second assertion of the theorem, assume that $d$ and $d_\ast$ are not roughly similar. By applying Proposition \ref{thm.stricconvexvsstraight}, this implies that the Manhattan curve $\theta_{d_\ast/d}$ is strictly convex everywhere. For $\alpha= \alpha(d_\ast/d)$ and $\xi=\xi(d_\ast/d)$ as above, we define the pseudo metric on $\G$ given by 
\[
d_\xi(x,y) =\xi d_\ast(x,y)+\thet_{d_\ast/d}(\xi)d(x,y)= \alpha d(x,y) + \xi(d_\ast(x,y) - d(x,y))
\] for $x,y\in \G$. By \cite[Lem.~4.1]{reyes} we have that $d_\xi \in \Dc_\G$. Also, by Proposition \ref{thm.stricconvexvsstraight} we necessarily have  $\alpha(d_\ast/d) <1$. Indeed, the Manhattan curve $\theta_{d_\ast/d}$ is strictly convex and goes through $(0,1)$ and $(1,0)$, so it is strictly below the line $s(t) = 1-t$ on $(0,1)$. From the definition of the Manhattan curve, we see that $d_\xi$ has exponential growth rate $1$. Suppose now that $f:\R \ra \R_{\geq 0}$ is non-decreasing and such that $f(t)=o(t)$ as $t\to \infty$. If $x\in \G$ and $T>0$ satisfy
$d(o,x) \leq T$ and $|d(o,x) - d_\ast(o,x)| \leq f(d(o,x))$, then $d_\xi(o,x) \leq \alpha T + \xi f(T)$, and so we obtain
\small \begin{align*}
    \#\{ x\in\G: d(o,x) \leq T, |d(o,x) - d_\ast(o,x)| \leq f(d(o,x)) \} &\le  \#\{ x \in \G: d_\xi(o,x) \le \alpha T + \xi f(T)\}\\
    &= O(e^{\alpha  T+\xi f(T)})
\end{align*}
\normalsize as $T\to\infty$. By Gromov hyperbolicity there is a constant $C'>0$ depending only on $d$ and $d_\ast$ such that for any $[x] \in \conj$ there is $y\in[x]$ such that both $|\ell_d[x] - d(o,y)| \le C'$ and $|\ell_{d_\ast}[x] - d_\ast(o,y)| \le C'$ hold (see e.g.~\cite[Lem.~3.2]{cantrell.tanaka.1}). It follows that
\[
\#\{[x] \in \conj : \ell_d[x] \leq  T, \ |\ell_d[x] - \ell_{d_\ast}[x]| \le f(\ell_d[x])  \} = O(e^{\alpha  T+\xi f(T)}). 
\]
From this, and after rearranging and using Lemma \ref{lem.sym}, we see that if the distance between $\ell_d$ and $\ell_{d_\ast}$ on a set $E \subset \conj$ is bounded by a function that is sublinear on $\ell_d$, then the exponential growth rate of $E$ with respect to $d$ is at most $1 - \beta(d_\ast,d)$, i.e.
\[
v_d(E) \le 1 - \beta(d,d_\ast).
\]
This proves the second assertion, and hence the theorem.
\end{proof}

Note that in the proof of Theorem \ref{thm.main} we also established the following result.
\begin{theorem}\label{thm.disp}
Let $\G$ be a non-elementary hyperbolic group and let $\beta$ be the function defined above.
    Suppose that $d,d_\ast \in \Dc_\G$ are two pseudo metrics with exponential growth rates $v_d, v_{d_\ast} >0$. Then the following are equivalent:
    \begin{enumerate}
        \item $d$ and $d_\ast$ are not roughly similar; and,
        \item $\beta(d,d_\ast) > 0$.
    \end{enumerate}
    Furthermore, when either (hence both) of these conditions are satisfied we have that
      \[
\limsup_{T\to\infty} \frac{1}{T} \log \left( \# \{x \in U: d(o,x) \leq  T \} \right) \le v_d - \beta(d,d_\ast)
\]
for any function $f:\R \ra \R_{\geq 0}$ such that $f(t)=o(t)$ as $t \to \infty$ and any subset $U \subset \G$ for which $|v_d d(o,x) - v_{d_\ast} d_\ast(o,x) |\leq f(d(o,x))$ for all $x\in U$.
\end{theorem}

Now we turn our attention to Theorem \ref{thm.continuous}. Recall the definition of $\overline{\beta}$ from the introduction.
\begin{definition}
    We extend $\beta : \Dc_\G \times \Dc_\G \to \R$ to $\overline{\beta} : \scrD_\G \times \scrD_\G \to \R$ by setting
    \[
    \overline{\beta}([d],[d_\ast]) =  \frac{\beta(d,d_\ast)}{v_d} = \beta(v_d d, v_{d_\ast} d_\ast).
    \]
\end{definition}

Note that $\overline{\beta}$ is well-defined: if $d$ and $d'$ are roughly similar then $v_d d$ and $v_{d'} d'$ are roughly isometric and so $\beta(v_d d, v_{d_\ast} d_\ast) = \beta(v_{d'}d', v_{d_\ast} d_\ast)$ and $\beta(v_{d_\ast} d_\ast, v_d d) = \beta(v_{d_\ast} d_\ast, v_{d'} d')$ for any $d_\ast \in \Dc_\G$. 

\begin{lemma}\label{lem.uc}
    Suppose that $[d_n] \to [d]$ in $(\scrD_\G, \Delta)$ as $n \to \infty$ and that $d_n, d, d_\ast \in \Dc_\G$ have exponential growth rate $1$ for each $n \ge 1$. Then $\theta_{d_\ast/d_n}$ and $\theta_{d_n/d_\ast}$ converge uniformly on compact sets to $\theta_{d_\ast/d}$ and $\theta_{d/d_\ast}$ respectively.
\end{lemma}
\begin{proof}
      By \cite[Thm.~1.1]{cantrell.reyes} there exist sequences $C_{0,n}, C_{1,n}, D_{n} >0$ such that $C_{0,n}, C_{1,n} \to 1$ as $n\to\infty$ with
    \[
    C_{0,n} d(o,x) -D_n \le d_n(o,x) \le C_{1,n}d(o,x) + D_n \  \text{ for all $x \in \G, n \ge 1$.}
    \]
    Now note that for $a,b \in \R$, the quantity $\exp( -ad_\ast(o,x) - bd_n(o,x))$ is either bounded above by $\exp(-ad_\ast(o,x) - b C_{0,n} d(o,x) + bD_n )$ and below by $\exp(-ad_\ast(o,x) - b C_{1,n} d(o,x) - bD_n )$ or vice versa depending on the sign of $b$. In either case it follows from the definition of the Manhattan curve that
    \[
    | \theta_{d_\ast/d_n}(a) - \theta_{d_\ast/d}(a)  | \le \max(|C_{0,n}^{-1} -1|, |C_{1,n}^{-1} -1 | ) |\theta_{d_\ast/d}(a)|
    \]
    and we see that $\theta_{d_\ast/d_n}$ converge uniformly on compact sets to $\theta_{d_\ast/d}$. The same method also shows that  $\theta_{d_n/d_\ast}$ converge uniformly on compact sets to $\theta_{d/d_\ast}$ and the proof is complete.
\end{proof}

\begin{proof}[Proof of Theorem \ref{thm.continuous}]
By definition 
\[
\overline{\beta}([d],[d_\ast]) = \sup_{t\in [0,1]} (1 - t - \theta_{v_{\ast} d_\ast/v_d d}(t)).
\]
From our above comments, we know that $\overline{\beta}$ is well-defined, and by Remark \ref{rem.symb} we see that it is symmetric in its variables, which settles (1).

To prove (2), suppose that $[d_n] \to [d]$ in $(\scrD_\G, \Delta)$ as $n\to\infty$. Without loss of generality we can assume that each $d,d_\ast, d_n$ have exponential growth rate $1$. By Lemma \ref{lem.uc} we have that $\theta_{ d_\ast,  d_n}$ converges uniformly to $\theta_{d_\ast/ d}$ on $[0,1]$ and so 
\[
\overline{\beta}([d_n],[d_\ast]) = \sup_{t\in [0,1]} (1 - t - \theta_{ d_\ast/ d_n}(t)) \to \sup_{t\in [0,1]} (1 - t - \theta_{ d_\ast/ d}(t))
= \overline{\beta}([d],[d_\ast])
\]
as $n\to\infty$.

Finally, to prove (3)
we fix two pseudo metrics $d,d_\ast \in \Dc_\G$ with exponential growth rate $1$.
    By strict convexity of the Manhattan curve $\theta_{d_\ast/d}$, the lines $s_1(t)= 1 - \Dil(d_\ast,d) t$ and $s_2(t) = \Dil(d,d_\ast)^{-1} - \Dil(d,d_\ast)^{-1} t$ lie beneath $\theta_{d_\ast/d}$ on $(0,1)$. In particular, we have that
     $\beta(d,d_\ast) \le \sup_{t\in[0,1]} (1 - t - F(t))$ where
    \[
     F(t) = \max(1 - \Dil(d_\ast,d) t, \Dil(d,d_\ast)^{-1} - \Dil(d,d_\ast)^{-1} t). 
    \]
    It follows, since the graphs of $s_1$ and $s_2$ intersect at
    \[
    \left( \frac{\Dil(d,d_\ast) -1}{\Dil(d,d_\ast)\Dil(d_\ast,d) -1 },  \frac{\Dil(d_\ast,d) -1}{\Dil(d,d_\ast)\Dil(d_\ast,d) -1 } \right),
    \]
    that
    \[
    \beta(d,d_\ast) \le \frac{(\Dil(d,d_\ast)-1)(\Dil(d_\ast,d)-1)}{\Dil(d,d_\ast)\Dil(d_\ast,d) - 1} = \frac{(\Dil(d,d_\ast)-1)(\Dil(d_\ast,d)-1)}{e^{\Delta([d],[d_\ast])} - 1}.
    \]
    To conclude the proof we note that
   \small \begin{align*}
      (\Dil(d,d_\ast)-1)(\Dil(d_\ast,d)-1) &\le \Dil(d,d_\ast)\Dil(d_\ast,d) - 2 \sqrt{\Dil(d,d_\ast)\Dil(d_\ast,d)} +1   \\
      &= \left( e^{\frac{1}{2}\Delta([d],[d_\ast]) } -1 \right)^2
    \end{align*}
  \normalsize by the arithmetic-geometric mean inequality. Strict convexity guarantees that this inequality is strict if and only if $d$ and $d_\ast$ are not roughly similar.
\end{proof}
Part (3) then shows that
    if $d,d_\ast \in \Dc_\G$ are not roughly similar and have exponential growth rate $1$ then for all $t \in (0,1)$ we have
    \[
    \frac{2}{e^{\frac{1}{2} \Delta([d],[d_\ast])}+1} <  t + \theta_{d_\ast/d}(t) < 1 \ \text{ and } \ 
    \alpha(d_\ast/d) \ge \frac{2}{e^{\frac{1}{2} \Delta([d],[d_\ast])}+1}.
    \]
    This provides estimates on the exponential growth rate for the correlation asymptotics of Schwartz and Sharp \cite{schwartz.sharp}, Dal'bo \cite{dalbo} and Cantrell \cite{cantrell} in terms of the symmetrized Thurston distance.


\section{Achieving the growth rate for word metrics and automorphisms}\label{sec.auto}
In this section we prove that the growth rate constant $\beta$ in Theorem \ref{thm.main} is optimal for pairs of metrics related by automorphisms of $\G$, or more generally, for word metrics satisfying a constraint on their growth rates.
If $S\subset \G$ is a finite, symmetric generating set with word metric $d_S$, we denote $\ell_S=\ell_{d_S}$ and $v_S=v_{d_S}$. We obtain the following.

\begin{theorem}\label{thm.word}
     Let $\G$ be a non-elementary hyperbolic group.
    Consider two finite, symmetric generating sets $S, S_\ast$ for $\G$ for which $v_S/v_{S_\ast}$ is rational. Then we have that
  \[
   \limsup_{T\to\infty} \frac{1}{T} \log (\#\left\{ [x] \in \conj: \ell_S[x] \leq T, v_S\ell_S[x] = v_{S_\ast}\ell_{S_\ast}[x]\right\} )= v_S - \beta(d_S,d_{S_\ast}).
   \]
\end{theorem}

The proof of Theorem \ref{thm.word} requires techniques from thermodynamic formalism and uses the Cannon automatic structure for hyperbolic groups. To avoid having to define all of these objects here, we refer the reader to \cite{cantrell}. Throughout the proof, we will follow same terminology and notation as that used in \cite{cantrell}. We require the following lemma which will also be useful in Subsection \ref{subsec.coame}.

\begin{lemma}\label{lem.count}
Given $d\in \calD_\G$ with exponential growth rate $v_d$ there exists $C\geq 0$ such that the following holds. For any $[x]\in \conj$ and $T>0$ we have
$$\log (\# \{y\in [x] : d(o,y) \leq T \})\leq \log(\ell_d[x]+C)+\frac{v_d(T-\ell_d[x])}{2}+C.$$
\end{lemma}
\begin{proof}
   Suppose that $d$ is $\al$-roughly geodesic. We will use the following facts, which are well-known when $d$ is induced by a geometric action on a hyperbolic geodesic space, and whose proofs can be adapted to the roughly geodesic setting. There exist constants $C_0,C_1,C_2,\al_0 \geq 0$ such that for any $[x]\in \conj$:
\begin{enumerate}
    \item there exists $x_0\in [x]$ such that $|d(o,x_0)-\ell_d[x]|\leq C_0$ \cite[Lem.~3.2]{cantrell.tanaka.1};
    \item if $x_0\in [x]$ is as in item (1) and $\gam \subset \G$ is an $\al$-rough geodesic joining $o$ and $x_0$, then $\widehat\gam=\bigcup_{k\in \Z}{x_0^k \gam}$ is an $x_0$-invariant $\al_0$-rough geodesic (note that $\# \gam \leq \ell_d[x]+\al+C_0$);
    \item if $x_0\in [x]$, $\gam$ and $\widehat\gam$ are as in items (1) and (2), then for any $a\in \G$ we have $|d(x_0a,a)-2d(a,\widehat\gam)-\ell_d[x]|\leq C_1$ (cf.~\cite[Lem.~5.10]{b.f}); and,
    \item for all $R>0$ and $u\in \G$ we have
    $\log(\# \{y\in \G : d(u,y) \leq R\})\leq v_dR+C_2$ 
 \cite[Cor.~6.8]{coornaert}.
\end{enumerate}
 Now we start the proof of the lemma, so fix $[x]\in \conj$ and $T>0$ and set $\ell=\ell_d[x]$. Let $x_0\in [x]$, $\gam$ and $\widehat\gam$ satisfy items (1)-(3) above. If $y=a^{-1}x_0a\in [x]$ is such that $d(o,y)=d(x_0a,a)\leq T$, then $d(a,\widehat\gam)\leq \frac{C_1}{2}+\frac{(T-\ell)}{2}$. After replacing $a$ by $x_0^ka$ for some $k$, we can assume that $d(a,u)\leq \frac{C_1}{2}+\frac{(T-\ell)}{2}$ for some $u\in \gam$. By item (4) we conclude
 \begin{align*}
     \log (\# \{y\in [x] : d(o,y)\leq T \}) &\leq \log \left(\# \left\{a\in \G : d(a,\gam)\leq \frac{C_1}{2}+\frac{(T-\ell)}{2}\right\}\right)\\
     & \leq \log \left(\sum_{u\in \gam}{\# \left\{a\in \G : d(a,u)\leq \frac{C_1}{2}+\frac{(T-\ell)}{2} \right\} }\right) \\
     & \leq \log(\# \gam)+v_d\left(\frac{C_1}{2}+\frac{(T-\ell)}{2} \right)+C_2 \\
     & \leq \log(\ell+\al+C_0)+v_d\left(\frac{C_1}{2}+\frac{(T-\ell)}{2} \right)+C_2. \qedhere
     \end{align*}
\end{proof}

\begin{proof}[Proof of Theorem \ref{thm.word}]
Without loss of generality assume that $d_S, d_{S_{\ast}}$ are not roughly similar. By Lemma 3.8 and Example 3.9 in \cite{calegari.fujiwara} we can find a Cannon coding for $(\G, S)$ with corresponding shift space $(\Sigma,\sigma)$ and a H\"older continuous function $r_\ast: \Sigma \to \mathbb{R}$ satisfying the following: if $z = (z_1, z_2,\ldots, z_n,z_1, \ldots)\in \Sigma$ is a periodic orbit of period $n$, then the $n$th Birkhoff sum of $r_\ast$ on $z$ outputs the value $\ell_{S_\ast}[x]$, where $x\in \G$ is the group element obtained by multiplying the labelings in the finite path $z_1,z_2, \ldots, z_n$. By construction, $r_\ast$ is constant on $2$ cylinders and takes values in $\mathbb{Z}$. Note that by the assumption that $v_{S_\ast}/v_S$ is rational, the function $v_{S_\ast} r_\ast/v_S - 1$ takes values in $(1/p) \mathbb{Z}$ for some $p\in \mathbb{Z}$. Define $\overline{r} = p(v_{S_\ast} r_\ast/v_S - 1)$ so that $\overline{r}$ takes values in $\mathbb{Z}$. 

We now fix a word maximal component $\mathcal{C}$ in the coding. Our aim is to apply a result of Pollicott and Sharp \cite[Thm.~1]{pollicott.sharp} to the pair $(\Sigma_\mathcal{C}, \overline{r})$. To do so, we need to verify two conditions:
\begin{enumerate}
    \item the function $\overline{r}$ is not cohomologous to a constant function; and,
    \item there exists a fully supported, shift-invariant probability measure $\mu$ on $\Sigma_\mathcal{C}$ such that 
    \[
    \int \overline{r} \ d\mu = 0.
    \]
\end{enumerate}
Note that $(1)$ is slightly different from the condition written in \cite{pollicott.sharp}. However, from inspecting the proof it is sufficient. For the first point: if $\overline{r}$ is cohomologous to a constant function then there exists $\tau \in \mathbb{R}$ such that $\ell_{S_*}[x] = \tau\ell_S[x]$ for all $[x]\in\conj$ by \cite[Cor.~3.6]{cantrell} and $d_S, d_{S_\ast}$ are roughly similar, which we are assuming not to be the case. The second point requires a little more work. By \cite[Lem.~3.3]{cantrell}, the Manhattan curve $\theta_{S_\ast/S}(t)$ for the pair $d_S, d_{S_\ast}$ is given by the pressure function $t \mapsto \text{P}_\mathcal{C}(-t r_\ast)$. We note that \cite[Lem.~3.3]{cantrell} is written for a strongly hyperbolic metric $d$ and word metric $d_S$, but the same argument works for our pair of word metrics $d_S, d_{S_*}$. Further, since $\theta_{S_\ast/S}'(0) \le -v_S/v_{S_\ast}$ and $\theta_{S_\ast/S}'(v_{S}) \ge -v_S/v_{S_\ast}$, we see that there must be a point $\xi \in (0,v_{S})$ such that $\theta_{S_\ast/S}'(\xi) = - v_S/v_{S_\ast}$. It follows from standard facts from thermodynamic formalism that 
\[
\int_{\Sigma_\mathcal{C}} -r_\ast \ d\mu_{-\xi r_\ast} = 
\frac{d\text{P}_\mathcal{C}(-tr_\ast)}{dt}\Big|_{t=\xi} =  \theta_{S_\ast/S}'(\xi)  =  -v_S/v_{S_\ast}
\]
where $\mu_{-\xi r_\ast}$ is the equilibrium state for $-\xi r_\ast$. Up to an algebraic manipulation, this implies point $(2).$

We can now apply \cite[Thm.~1]{pollicott.sharp} to deduce that there exist $m\in\mathbb{Z}_{\ge 0}$, $C>0$ and $\beta' >0$ such that
\[
\# \left\{z\in \Sigma \colon \sigma^{mn}(z)  = z, \overline{r}^{mn}(z) = 0 \right\} \ge C {e^{\beta' mn}}{(mn)^{-3/2}} \ \text{ for each $n\ge 1$.}
\]
We note that Pollicott and Sharps's result is written for mixing subshifts of finite type and our subshift $\Sigma_\mathcal{C}$ may only be transitive. This is not an issue for us, since we can simply apply the result to $(\Sigma_\mathcal{C}, \sigma^{p})$ where $p$ is the period of the component.
By inspecting the proof in \cite{pollicott.sharp} and using \cite[Lem.~3.3]{cantrell} we see that $\beta'$ is precisely the value $\alpha(d_S, d_{S_\ast})$ introduced in Section \ref{sec.lsos} (see \cite{sharp.manhat} for more details).
Now note that each periodic orbit in $\Sigma_\mathcal{C}$ has a corresponding conjugacy class in $\G$. Furthermore, by Lemma \ref{lem.count}, there exists $C>0$ (independent of $T$) such that each conjugacy class of $\ell_S$ length $T$ has at most $C T+C$ corresponding periodic orbits in $\Sigma_\mathcal{C}$. We deduce that there exists $\widehat{C} >0$ and $m \in \mathbb{Z}_{\ge 0}$ such that
\[
\#\left\{ [x] \in \conj: \ell_S[x] = mn , v_S\ell_{S}[x] = v_{S_\ast}\ell_{S_\ast}[x]\right\} \ge \widehat{C} e^{mn\alpha(d_S,d_{S_\ast})} (mn)^{-5/2}
\]
for all $n\geq 0$. By Lemma \ref{lem.sym} this provides the required lower bound and concludes the proof.
\end{proof}
\begin{remark}
    It seems sensible to conjecture that this result holds for all hyperbolic groups and pairs of finite generating sets $S, S_\ast$ if we relax the condition that $v_S\ell_S[x] = v_{S_\ast}\ell_{S_\ast}[x]$ to $|v_S\ell_S[x] - v_{S_\ast}\ell_{S_\ast}[x]| \le \epsilon$ for some fixed $\epsilon>0$. This problem is related to large deviations with shrinking intervals. We note it is possible to prove this for pairs of \textit{strongly hyperbolic} metrics with an independence assumption by following the proof of \cite[Thm.~1.10]{cantrell}.
\end{remark}
We immediately obtain the following corollary. Given a finite, symmetric generating set $S\subset \G$ and automorphisms $\phi,\psi \in \Aut(\G)$, we will write $\beta_S(\phi,\varphi) := \beta(d_S^\phi, d_S^\varphi)$ where $d_S^\phi(x,y) = |\phi(x)^{-1}\phi(y)|_S$ and $d_S^\varphi(x,y) = |\varphi(x)^{-1}\varphi(y)|_S$.
Note that $d_S^\phi=d_{\phi^{-1}(S)}$ and $v_{\phi^{-1}(S)}=v_S$ for all $S$ and $\phi$.
\begin{corollary}\label{coro.betaauto}
      Let $\Gamma$ be a non-elementary hyperbolic group equipped with a finite and symmetric generating set $S$.
    For any two automorphisms $\phi,\varphi \in \textnormal{Aut}(\G)$ we have
    \[
    \limsup_{T\to\infty} \frac{1}{T} \log( \#\left\{ [x] \in \conj: \ell_S[\phi(x)] \leq T, \ell_S[\phi(x)]= \ell_{S}[\varphi(x)]\right\}) = v_S - \beta_S(\phi,\varphi).
    \]
\end{corollary}

\begin{remark}
   This corollary complements the results in \cite{car} and \cite{kap} for the case of free groups by exhibiting the largest non-spectrally rigid subset for pairs of automorphisms.
\end{remark}


\section{Marked length spectrum domination}\label{sec.mlsdom}
In this section we prove Theorems \ref{thm.main2} and \ref{thm.domination}. For both results, the key ingredient is the following lemma, which will follow by adapting the proof of \cite[Thm.~4.23]{cantrell.tanaka.1}.
\begin{lemma}\label{lem.ineq}
    Let $d\in \calD_\G$ and $E\subset \conj$ be  such that $v_d(E)=v_d$. Suppose $d_\ast\in \calD_\G$ is such that there exist $\eta_2\geq \eta_1>0$ and a function $f : \mathbb{R} \to \mathbb{R}_{\geq 0}$ satisfying $f(t) = o(t)$ as $t\to \infty$ and 
    $$\eta_1\ell_d[x]-f(\ell_d[x])\leq \ell_{d_\ast}[x] \le \eta_2\ell_d[x]+f(\ell_d[x])$$ for all $[x] \in E$. Then we have $$\eta_1 \leq \tau(d_\ast/d)\leq \eta_2.$$
\end{lemma}

\begin{proof}
    Suppose $d,d_\ast \in \Dc_\G$, $E \subset \conj$ and $\eta_1,\eta_2$ satisfy the assumptions in the lemma. Given $R>0$ we write $\mathcal{C}(d,n,R)= \{[x] \in \conj: |\ell_d[x] - n| \le R \}$. Then for any fixed $R > 0$ sufficiently large we have
    \begin{equation}\label{eq.density}
        \limsup_{n \to\infty} \frac{1}{n} \log \left(\#(E \cap \mathcal{C}(d,n,R))\right) = v_d.
    \end{equation}
    
    Also, consider the function
    \[
    I(s) = \theta_{d_\ast/d}(0) + \sup_{t\in\mathbb{R}} \{ts - \theta_{d_\ast/d}(-t)\},
    \]
    which is finite and lower semi-continuous on the interval $(\Dil(d,d_\ast)^{-1},\Dil(d_\ast,d))$.
    It follows from Proposition \ref{thm.stricconvexvsstraight} that $I(s) \ge 0$ with equality only when $s = \tau(d_\ast/d)$.
    By the same proof used to show \cite[Thm.~4.23]{cantrell.tanaka.1} we can prove the following. For any fixed $R$ sufficiently large
    \begin{equation}\label{eq.ldp}
          \limsup_{n\to\infty} \frac{1}{n} \log\left(\#\left\{ [x] \in \mathcal{C}(d,n,R) : \frac{\ell_{d_\ast}[x]}{\ell_d[x]} \in V \right\} \right) \le v_d - \inf_{s \in V} I(s)
    \end{equation}
    for any closed set $V \subset \mathbb{R}$. The result \cite[Thm.~4.23]{cantrell.tanaka.1} is analogous to \eqref{eq.ldp} except the counting is done over group elements and for the displacement functions instead of counting over conjugacy classes for the translation length functions.  To prove the above result we simply need to combine Proposition 3.1, Corollary 2.8 and the proof of Theorem 4.23 in \cite{cantrell.tanaka.1}. We leave the details to the reader.
    
    Now assume that $\tau(d_\ast/d) \notin [\eta_1,\eta_2]$ and choose a closed interval $V= [\eta_1-\ep,\eta_2+\ep]$ with $\ep>0$ small enough so that $\tau(d_\ast/d) \nin V$. By the assumptions placed on $E$ we have that 
    \[
    \frac{\ell_{d_\ast}[x]}{\ell_d[x]} \in V \ \text{for all by finitely many $[x] \in E$. }
    \]
    It then follows from \eqref{eq.ldp} that the growth rate of $E$ with respect to $d$ is strictly less than $v_d$. This directly contradicts \eqref{eq.density} and so we conclude that $\tau(d_\ast/d) \in [\eta_1,\eta_2]$.
\end{proof}

We now move onto the proofs of Theorems \ref{thm.main2} and \ref{thm.domination}.

\begin{proof}[Proof of Theorem \ref{thm.main2}]
    Let $d,d_\ast\in \calD_\G$ and $E\subset \conj$ be as in the statement of the theorem. By Lemma \ref{lem.ineq} we deduce $\tau(d_\ast/d) = 1$. The same lemma implies $\tau(d/d_\ast) = 1$. Indeed, note that since $d,d_\ast$ are quasi-isometric, condition (3) in the theorem holds when $d$ and $d_\ast$ are switched.
    Therefore
    \[
    \theta_{d_\ast/d}'(0) = - \tau(d_\ast/d) = -1 = - \tau(d/d_\ast)^{-1} = \theta'_{d_\ast/d}(v_d) 
    \]
    and so by convexity, $\theta_{d_\ast/d}$ is a straight line on $(0,v_d)$.
    Hence by $(3)$ in Theorem \ref{thm.manprops} we see that $d,d_\ast$ are roughly isometric and the proof is complete.
\end{proof}

\begin{proof}[Proof of Theorem \ref{thm.domination}]
Take $d, d_\ast, E$ and $f$ as in the statement of the theorem. We would like to deduce that $\Dil(d_\ast,d) \le 1$.

    Given $a < 0$ there is a pseudo metric $d_a \in \Dc_\G$ that is within bounded distance of  $ad_\ast + \theta_{d_\ast/d}(a) d$ \cite[Prop.~4.1]{cantrell.reyes}. By assumption (2) we have that 
    \begin{equation}\label{eq.fix}
    \ell_{d_\ast}[x] \le \ell_d[x]+f(\ell_d[x]) \ \text{ on $E$.}      
    \end{equation}
    There are now two cases to consider:
    
   \noindent  \textit{Case 1: $\theta_{d_\ast/d}(a) + a \le 0$ for some $a<0$.}
By \cite[Rmk.~4.14]{cantrell.reyes} we also have that
\[
0<a \Dil(d_*,d)+\theta_{d_\ast/d}(a)
\]
and combining these two inequalities gives
\[
0 < a \Dil(d_*,d)+\theta_{d_\ast/d}(a) \le a \Dil(d_*,d) - a = a(\Dil(d_\ast,d) -1)
\]
implying that $\Dil(d_\ast,d) < 1$ as required.

\noindent  \textit{Case 2: $\theta_{d_\ast/d}(x) + a  > 0$ for all $a<0$.}
In this case expression $(\ref{eq.fix})$ implies that
    \[
\ell_{d_\ast}[x] \le \frac{\ell_{d_a}[x]}{a + \theta_{d_\ast/d}(a)}+f_a(\ell_{d_a}[x])
    \]
    for all $[x] \in E$ where $f_a(t)=o(t)$ as $t\to \infty$, and so $\tau(d_\ast/d_a) \le (a+\theta_{d_\ast/d}(a))^{-1}$ by Lemma \ref{lem.ineq} and assumption (2). 
 We now note that, by the definition of the Manhattan curve, the quantities $\theta_{d_\ast/d}(s)$ and $\theta_{d_\ast/d_a}(s)$ are related by the following equation
    \[
    \theta_{d_\ast/d}(s+\theta_{d_\ast/d_a}(s)) = \theta_{d_\ast/d_a}(s)\theta_{d_\ast/d}(a) \ \text{ for each $s \in \R$}.
    \]
    Differentiating this equation at $s=0$ gives the following expression
   \[
\theta_{d_\ast/d_a}'(0) =  -\left( a -  \frac{\theta_{d_\ast/d}(a)}{\theta_{d_\ast/d}'(a)}\right)^{-1},
   \] 
   and so
   \[
   \left( a -  \frac{\theta_{d_\ast/d}(a)}{\theta_{d_\ast/d}'(a)}\right)^{-1} =   -\theta_{d_\ast/d_a}'(0) = \tau(d_\ast/d_a) \le (a+\theta_{d_\ast/d}(a))^{-1}.
   \]
   This rearranges to give $\theta_{d_\ast/d}'(a) \ge -1$. Since this is true for all $a <0$, by \cite[Cor.~3.3]{cantrell.tanaka.1} we deduce that
   \[
   \Dil(d_\ast,d) = -\lim_{a\to-\infty} \theta_{d_\ast/d}'(a) \le 1,
   \]
    which concludes the proof of the second case and also the theorem.
\end{proof}


\section{Applications}\label{sec.applications}

In this section, we prove the results that appeared in the introduction.

\subsection{Length spectrum rigidity from subgroups/subsets}\label{subsec.coame}
In this subsection we prove Corollary \ref{cor.1}. As we mentioned in the introduction, a subgroup $H < \G$ is \emph{co-amenable} if the action of $\G$ on the left coset space $H \backslash \G$ admits a $\G$-invariant mean. For example, a normal subgroup $H\trianglelefteq \G$ is co-amenable if and only if the quotient group $\G/H$ is amenable. In particular, the commutator subgroup $[\G,\G]$ is co-amenable. 

In \cite[Thm.~1.1]{CDST}, Coulon, Dougall, Shapira and Tapie gave a characterization of co-amenability in terms of exponential growth rates. For $\G$ hyperbolic, it can be stated as follows: given any finite and symmetric generating set $S\subset \G$ with word metric $d=d_S$, a subgroup $H<\G$ is co-amenable if and only if
\begin{equation}\label{eq.coamenable}
v_d(H):=\limsup_{T\to\infty}{ \frac{1}{T} \log (\#\{x \in H: d(o,x) \leq T \})} = v_d.
\end{equation}

To prove our corollaries, we need to relate the exponential growth rate of the displacement for subsets of $\G$ and the exponential growth rate for the translation length for the corresponding subset of conjugacy classes. That relation is given by the following proposition.

\begin{proposition}\label{prop.samegrowthrate}
Let $d\in \calD_\G$ and let $A\subset \G$ be a set that is closed under taking conjugates. If $[A]\subset \conj$ denotes the set of conjugacy classes of elements in $A$, then
\begin{equation}\label{eq.growrate}
\limsup_{T\to\infty}{ \frac{1}{T} \log( \#\{y \in A: d(o,y) \leq T \})}=v_d([A]).
\end{equation}
\end{proposition}

\begin{proof}
Let $v_d(A)$ be the limit on the left-hand side of \eqref{eq.growrate}. Since $A$ is invariant under conjugacy, by item (1) in the proof of Lemma \ref{lem.count} we deduce $v_d([A])\leq v_d(A)$.

For the reverse inequality, we consider $\ep>0$, for which there is some $C_\ep>1$ such that 
$$\# \{ [x]\in [A]: \ell_d[x]< \ell+1\}\leq C_\ep e^{\ell(v_d([A])+\ep)}$$
for all $\ell\geq 0$. Also, by Lemma \ref{lem.count} there is a constant $C\geq 0$ such that for any integer $T>0$ we have
\begin{align*}
    \# \{y\in A : d(o,y)< T \} & = \sum_{\ell=0}^{T}{\left( \sum_{[x]\in [A]: \ell\leq \ell_d[x]<\ell+1}{\# \{y\in [x]: d(o,y) < T\} }\right) }\\ 
    & \leq e^C\cdot (T+C)\cdot \left(\sum_{\ell=0}^{T} {\# \{[x]\in [A]: \ell_d[x]<\ell+1 \}\cdot e^{v_d\left(\frac{T-\ell}{2} \right)}} \right).
\end{align*}
In addition, for any integer $T>0$ we have 
\begin{align*}
    \sum_{\ell=0}^{T} {\# \{[x]\in [A]: \ell_d[x]< \ell+1 \}\cdot e^{v_d\left(\frac{T-\ell}{2} \right)}} & \leq C_\ep\cdot e^{v_dT/2}\cdot \sum_{\ell=0}^{T}{e^{\ell(v_d([A])+\ep-v_d/2)}} \\
    & \leq \k_\ep e^{T(v_d([A])+\ep)}
\end{align*}
for some $\k_\ep>1$ depending only on $\ep$, $v_d$ and $v_d([A])$. Combining these two inequalities we conclude $v_d(A)\leq v_d([A])+\ep$ for all $\ep>0$, which implies the desired inequality.
\end{proof}

We now prove Corollary \ref{cor.1}.

\begin{proof} [Proof of Corollary \ref{cor.1}]
Let $H<\G$ be a co-amenable subgroup and let $E=\{[x]\in \conj:x\in H\}$.
If $d\in \calD_\G$ is a word metric for a finite generating set with exponential growth rate $v_d$, from \eqref{eq.coamenable} we have that $v_d(H)=v_d$. Also, by \cite[Lemma 5.1]{or} we know that metric structures induced by word metrics are dense in $(\scrD_\G,\Del)$. Since the function $[d]\to v_d(H)/v_d$ is well-defined and continuous on $(\scrD_\G,\Del)$, we deduce that equation \eqref{eq.coamenable} is valid for \emph{all} $d\in \calD_\G$. This observation combined with Proposition \ref{prop.samegrowthrate} tells us that $v_d(E)=v_d(H)=v_d$ for any $d\in \calD_\G$. Finally, if $d,d_\ast\in \calD_\G$ satisfy
$$d_\ast(o,x)\leq d(o,x)+f(d(o,x))$$
for all $x\in H$ and $f:\R \ra \R_{\geq 0}$ such that $f(t)=o(t)$ as $t \to \infty$, then we have $$\ell_{d_\ast}[x]\leq \ell_d[x]$$ for all $[x]\in E$. The conclusion then follows by Theorems \ref{thm.main2} and \ref{thm.domination}.
\end{proof}


\begin{remark}
Incidentally, from the proof of Corollary \ref{cor.1} we deduce that if $\G$ is non-elementary hyperbolic and $\al$ is a homology class in $\G$ with set of conjugacy classes $[\al]$, then for any $d\in \calD_\G$ with exponential growth rate $v_d$ we have
$$v_d([\al]) =  v_d.$$
This partially extends a result of Adachi and Sunada \cite[Thm.~1]{adachi.sunada}, where they proved it when $\G$ is the fundamental group of a compact manifold and $d$ is induced by a Riemannian metric with Anosov geodesic flow.
\end{remark}

The results obtained so far can be used to compare pseudo metrics induced by automorphisms of $\G$. Given two automorphisms $\phi, \psi \in \text{Aut}(\G)$ of $\G$ and $d\in \calD_\G$, it is natural to ask when $\phi$ and $\psi$ displace the elements of $\G$ in a similar way. One way to parse this is to ask whether the pseudo metrics $d^\psi(x,y) = d(\psi(x),\psi(y))$ and $d^\phi(x,y) = d(\phi(x),\phi(y))$ are roughly isometric (i.e. they have the same marked length spectrum). There has been much interest in these sorts of problems in the setting of Culler-Vogtmann Outer Space, see \cite{car}, \cite{kap}, \cite{ray}. Our results apply to this setting and provide various length spectrum rigidity results. For example, we can deduce the following.
Let $\phi, \psi$ be two automorphisms 
of the free group on $n$ generators $F_n$. Equip $F_n$ with a finite generating set $S$ and let $H = [F_n, F_n]$ be the commutator subgroup. For $T \ge 1$ let $H_T$ be the collection of elements in $H$ of $S$ translation length at most $T$, i.e. $H_T = \{ x \in H:  \ell_S[x] \le T\}$. If $d_{\phi^{-1}(S)}$ and $d_{\psi^{-1}(S)}$ are not roughly isometric, then the supremum $\sup_{x \in H_T}|\ell_S[\phi(x)]-\ell_S[\psi(x)]|$ must grows at least linearly as $T\to\infty$.


\subsection{Volume rigidity}\label{subsec.area}

We now prove our results regarding volume rigidity. We start with the proof of Corollary \ref{cor.area}.
\begin{proof}[Proof of Corollary \ref{cor.area}]
    By \cite[Thm.~1.1]{croke.dairbekov}, if $\ell_{\mathfrak{g}_\ast}[x] \le \ell_{\mathfrak{g}}[x]$ for all conjugacy classes $[x]$ in $\pi_1(S)$, then $\text{Area}(S,\mathfrak{g}_\ast) \le \text{Area}(S,\mathfrak{g})$. By Corollary \ref{cor.1}, if the inequality $\ell_{\mathfrak{g}_\ast}[x] \le \ell_{\mathfrak{g}}[x]$ holds on a co-amenable subgroup (resp.~homology class) then it holds everywhere and we are done.
\end{proof}

Similar results can be deduced in higher dimensions. For example, the same argument as above combined with the volume inequality of Guillarmou and Lefeuvre \cite[Thm.~2]{guilef} gives Theorem \ref{thm.volgl}.

Outside the manifold setting, we can recover area inequalities in the context of Fuchsian buildings \cite{bou}. These are a class of polygonal 2-complexes supporting locally $\CAT(-1)$ metrics. When $X$ is the quotient of a Fuchsian building by a subgroup of the combinatorial automorphism group, we let $\mathcal{M}_{neg}(X)$ consists of all the (locally) negatively curved, piecewise Riemannian metrics on $X$. For these quotients, we can promote a result of Constantine and Lafont \cite[Thm.~1.3]{conlaf} and deduce the following.

\begin{corollary}\label{cor.areabuilding}
    Let $X$ be the quotient of a Fuchsian building
    by a subgroup of the combinatorial automorphism group, which we assume acts freely and cocompactly. Let $\gm,\gm_\ast\in \mathcal{M}_{neg}(X)$, and suppose that $\ell_{\mathfrak{g}_\ast}[x] \le \ell_{\mathfrak{g}}[x]$ for each conjugacy class $[x]$ represented by an element in a fixed co-amenable subgroup (resp. homology class) of $\pi_1(X)$. Then 
    \[
    \textnormal{Area}(X,\mathfrak{g}_\ast) \le \textnormal{Area}(X,\mathfrak{g}).
    \]
\end{corollary}

Finally, we apply our results to the volumes of Hitchin representations. Let $S$ be a closed orientable surface with hyperbolic fundamental group $\G$ and fix $m\geq 2$. A representation of $\G$ into $\PSL_m(\R)$ is $m$-\emph{Fuchsian} if it is the composition of a Fuchsian representation of $\G$ into $\PSL_2(\R)$ and an irreducible representation of $\PSL_2(\R)$ into $\PSL_m(\R)$. A representation $\rho:\G \ra \PSL_m(\R)$ is \emph{Hitchin} if it can be continuously deformed to an $m$-Fuchsian representation. If $\rho$ is a Hitching representation of $\G$, then we can associate an (asymmetric) geodesic current $\omega_\rho$ on $\G$, called the \emph{Liouville current} of $\rho$. Among other properties, this current satisfies
$$i([x],\omega_\rho)=\log\left(\frac{\lam_1 (\rho(x))}{\lam_m(\rho(x))}\right),$$
where $\lam_1(A)$ (resp. $\lam_m(A)$) denotes the largest (resp. smallest) modulus of an eigenvalue of $A$, $[x]\in \conj$ is seen as the corresponding rational geodesic current on $\G$, and $i$ denotes Bonahon's intersection number for currents \cite{bonahon.currentsTeich}. In \cite{BCLS18}, the \emph{Liouville volume of} the Hitchin representation $\rho$ was defined as $$\textnormal{vol}_L(\rho)=i(\omega_\rho,\omega_\rho).$$  

We can use Corollary \ref{cor.1} to refine the volume rigidity for Liouville currents by Bridgeman, Canary, Labourie and Sambarino \cite[Thm.~1.4]{BCLS18}.

\begin{proof}[Proof of Corollary \ref{cor.areahitchin}]
    Let $\|\cdot\|$ be the Euclidean norm on $\PSL_m(\R)$. For a Hitchin representation $\rho:\G \ra \PSL_m(\R)$, consider the non-negative functions $\psi^\rho,\widehat{\psi}^\rho$ on $\G \times \G$ given by 
    $$\psi^\rho(x,y)=\log \|\rho(x^{-1}y)\|, \hspace{2mm}\widehat{\psi}^\rho(x,y)=\log \|\rho(y^{-1}x)\|.$$
    Since Hitchin representations are projective Anosov (i.e. 1-dominated), their contragradient representations are also projective Anosov. 
    It follows from \cite[Lem.~3.14]{cantrell.reyes} that $d_\rho:=\psi^\rho+\widehat{\psi}^\rho$ is in fact a metric $d_\rho \in \calD_\G$. Therefore, if $\rho,\rho_\ast$ are Hitching representations of $\G$ such that $\lam_1(\rho_\ast(x))\leq \lam_1(\rho(x))$ for all $x$ belonging to the co-amenable subgroup $H<\G$ and since $H$ is closed under taking inverses, we deduce $$\ell_{d_{\rho_\ast}}[x]\leq \ell_{d_\rho}[x]$$
    for every $x$ in $H$, and hence for every $[x]\in \conj$ by Corollary \ref{cor.1}. The conclusion then follows by \cite[Thm.~1.4]{BCLS18} and noting that 
    $$\ell_{d_\rho}[x]=\log\left(\frac{\lam_1 (\rho(x))}{\lam_m(\rho(x))}\right)$$
    for every $[x]\in \conj$ and every Hitching representation $\rho$.

If instead we assume that the spectral radius inequality holds for all $x$ in a fixed homology class, then the conclusion still holds. Indeed, all of the theorems in Subsection \ref{sec.ros}, as well as their corollaries, also have versions for pairs of projective Anosov representations. That is, we can replace $d,d_\ast$ with the log-norm functions for Anosov representations $\rho:\G \to \PSL_m(\R), \rho_\ast:\G \to \PSL_{m_\ast}(\R)$. In this setting, the logarithm of the spectral radii $ \log\lambda_1(\rho(x)), \log\lambda_1(\rho_\ast(x))$ play the role of the stable translation length functions. The proofs are the same as for the pseudo metrics in $\Dc_\G$, since for these hyperbolic distance-like functions we have a good theory of Manhattan curves \cite[Sec.~7.3]{cantrell.tanaka.2}. We leave the details to the reader.
\end{proof}




\subsection{Growth rate of quasi-convex subgroups and more} \label{sec.gr}
In this subsection we apply our results to study growth rates of quasi-convex subgroups. Our main result is Theorem \ref{thm.growthrateE}, from which we deduce Theorems \ref{thm.gapword} and \ref{thm.growthtrees} from the introduction. We first recall the notion of boundary pseudo metrics \cite{cantrell.reyes}.

\begin{definition}\label{def.boundarypseudo}
 The space of \emph{(Manhattan) boundary pseudo metrics} on $\G$ is $\partial \calD_\G:=\ov\calD_\G \bs \calD_\G$. Equivalently, this is the set of all pseudo metrics in $\ov\calD_\G$ that are not quasi-isometric to pseudo metrics in $\calD_\G$.  
\end{definition}

\begin{theorem}\label{thm.growthrateE}
   Given $\widehat{d} \in \partial \calD_\G$ and $L\geq 0$, let $E=\{[x]\in \conj : \ell_{\widehat{d}}[x]\leq L\}$. Then for any $d\in \calD_\G$ with exponential growth rate $v_d$ we have
    \begin{equation*}
        v_d(E)< v_d.
    \end{equation*}
\end{theorem}

This result already implies Theorem \ref{thm.growthtrees} since for $\G$ either a free or surface group, small actions on $\R$-trees induce pseudo metrics in $\partial \calD_\G$ \cite[Thm.~1.6 (3)]{cantrell.reyes}.

Before proving Theorem \ref{thm.growthrateE}, we deduce the following corollary, which immediately implies Theorem \ref{thm.gapword}.

\begin{corollary}\label{cor.dfw}
    Let $\G$ be a non-elementary hyperbolic group and let $d\in \calD_\G$ be any pseudo metric with exponential growth rate $v_d$. Then for any infinite index quasi-convex subgroup $H<\G$ we have
    \begin{equation*}
        \limsup_{T\to \infty}{\frac{1}{T} \log (\#\{x\in H : d(o,x) \leq T\})}<v_d.
    \end{equation*}
\end{corollary}

\begin{proof}[Proof of Corollary \ref{cor.dfw}]
    Let $d\in \calD_\G$ and $H<\G$ be an infinite quasi-convex subgroup of infinite index. By Proposition \ref{prop.samegrowthrate} and Theorem
    \ref{thm.growthrateE}, it is enough to construct a pseudo metric $\widehat{d}\in \partial \calD_\G$ such that $H\subset A:=\{x\in \G : \ell_{\widehat{d}}[x]=0\}$. To this end, let $S\subset \G$ be a finite, symmetric generating set, and consider the coned-off Cayley graph $\Cay(\G,S,H)$ with the corresponding graph metric. If $\widehat{d}$ is any pseudo metric on $\G$ induced by the action of $\G$ on $\Cay(\G,S,H)$, then $\widehat{d}\in \partial \calD_\G$ since $H$ is infinite and of infinite index \cite[Prop.~6.14]{cantrell.reyes}. Moreover, the action of $H$ on $\Cay(\G,S,H)$ is by elliptic isometries, and hence $H\subset A$.
\end{proof}

\begin{proof}[Proof of Theorem \ref{thm.growthrateE}]
Let $d\in \calD_\G$, $\widehat{d}\in \partial \calD_\G$, and $E\subset \conj$ be as in the statement, and let $v_d$ be the exponential growth rate of $d$. 
Assume, for the sake of contradiction that $v_d(E)=v_d$, and consider the pseudo metric $d_\ast:=d+\widehat{d}$. This pseudo metric belongs to $\calD_\G$ by \cite[Lem.~5.3]{cantrell.reyes}. Note that $\ell_{d_{\ast}}=\ell_{d}+\ell_{\widehat{d}}$, so that $|\ell_{d_{\ast}}[x]-\ell_{d}[x]|\leq L$ for all $[x]\in E$. On the other hand, the assumption $\widehat{d}\in \partial \calD_\G$ implies that $\ell_{\widehat{d}}$ is non-zero, and hence $d$ and $d_\ast$ are not roughly isometric. 
Since $d\leq d_\ast$ we also have 
$$v_{d_\ast}\leq v_d =v_d(E)=v_{d_\ast}(E)\leq v_{d_\ast}.$$
In particular we obtain $v_{d_\ast}(E)=v_{d_\ast}$, so that $d,d_\ast$ and $E$ satisfy conditions (1)-(3) of Theorem \ref{thm.main2}. This is the desired contradiction since $d$ and $d_\ast$ are not roughly isometric, which completes the proof of the proposition.
\end{proof}


\subsection{Uniform local marked length spectrum rigidity}

We conclude the article by proving Corollary \ref{cor.main}. Given $\del\geq 0$, we let $\scrD^\delta_\G \subset \scrD_\G$ denote the collection of metric structures represented by a $\delta$-hyperbolic pseudo metric with exponential growth rate $1$. In \cite[Thm.~1.5]{reyes}, the second author proved that when non-empty, $\scrD_\G^\del$ is a proper subspace of $\scrD_\G$ for every $\del\geq 0$.

\begin{proof}[Proof of Corollary \ref{cor.main}]
By Theorem \ref{thm.continuous} we know that $\overline{\beta}$ is continuous in each variable. Since $\scrD_\G^\del$ is proper, it follows that for any $d \in \Dc_\G$ and $D_1, D_2 >0$ the set $Y([d], D_1,D_2) = \{ [d_\ast] \in \scrD^\delta_\G : D_1 \le \Delta([d],[d_\ast]) \le D_2 \}$ is compact. The result follows from these observations and Lemma \ref{lem.eqality}, which guarantees that $\overline{\beta}([d],[d_\ast])$ is non-zero for $[d_\ast]\in Y([d], D_1, D_2)$.
\end{proof}





\noindent\small{Department of Mathematics, 
University of Warwick,
Coventry, CV4 7AL, UK}\\
\small{\textit{Email address}: \texttt{stephen.cantrell@warwick.ac.uk}\\
\\
\small{Department of Mathematics, Yale University, New Haven, CT 06511, USA}\\
\small{\textit{Email address}: \texttt{eduardo.c.reyes@yale.edu}}\\


\begin{thebibliography}{100}
\bibitem{adachi.sunada} T.~Adachi, T.~Sunada, Homology of closed geodesics in a negatively curved manifold. \textit{J. Differential Geom.} \textbf{26} (1987), 81--99.

\bibitem{BL}
A.~Bankovic, C.~J.~Leininger, Marked-length-spectral rigidity for flat metrics. \textit{Trans. Amer. Math. Soc.} \textbf{370} (2018), no. 3, 1867--1884.





\bibitem{bonahon}
F.~Bonahon, Geodesic currents on negatively curved groups. \textit{Arboreal group theory} (Berkeley, CA, 1988), 143–168, Math. Sci. Res. Inst. Publ., \textbf{19}, Springer, New York, 1991.

\bibitem{bonahon.currentsTeich}
F.~Bonahon, The geometry of Teichmüller space via geodesic currents. \textit{Invent. Math.} \textbf{92} (1988), 139--162.

\bibitem{bou}
M.~Bourdon, Sur les immeubles Fuchsiens et leur type de quasi-isom\'etrie. \textit{Ergodic Theory Dynam. Systems} \textbf{20} (2000), 343--364.


\bibitem{bowen}
R.~Bowen, Hausdorff dimension of quasi-circles. \textit{Publ. Math. IHES}, \textbf{50} (1979), 11--25.


\bibitem{b.f}
E.~Breuillard, K.~Fujiwara, On the joint spectral radius for isometries of non-positively curved spaces and uniform growth. \textit{Ann. Inst. Fourier} \textbf{71} (2021), no. 1, 317--391.

\bibitem{bridgeman.canary} 
M.~Bridgeman, R.~Canary, Simple length rigidity for Kleinian surface groups and applications. \textit{Comment. Math. Helv.} \textbf{92} (2017), 715--750.

\bibitem{BCLS18}
M.~Bridgeman, R.~Canary, F.~Labourie and A.~Sambarino, Simple root flows for Hitchin representations. \textit{Geom. Dedicata} \textbf{192} (2018), 57--86.



\bibitem{calegari.fujiwara}
D.~Calegari, K.~Fujiwara, Combable functions, quasimorphisms, and the central limit theorem. \textit{Ergodic Theory Dynam. Systems} \textbf{30} (2010), no. 5, 1343--1369.

\bibitem{cantrell} S.~Cantrell. Mixing of the Mineyev flow, orbital counting and Poincar\'e series for strongly hyperbolic metrics. \url{https://arxiv.org/abs/2210.11558}, arXiv preprint, 2022.

\bibitem{cantrell.reyes} S.~Cantrell, E.~Reyes, Manhattan geodesics and the boundary of the space of metric structures on hyperbolic groups, to appear in Commentarii Mathematici Helvetici, arXiv:2210.07136, 2022.

\bibitem{cantrell.tanaka.2} 
S.~Cantrell, R.~Tanaka, Invariant measures of the topological flow and measures at infinity on hyperbolic groups.
\textit{J. Mod. Dyn.} \textbf{20} (2024), 215--274.

\bibitem{cantrell.tanaka.1} 
S.~Cantrell, R.~Tanaka, The Manhattan curve, ergodic theory of topological flows and rigidity, to appear in Geometry \& Topology, arXiv:2104.13451, 2021.

\bibitem{car} 
M.~Carette, S.~Francaviglia, I.~Kapovich and A.~Martino, Spectral rigidity of automorphic orbits in free groups. \textit{Algebr. Geom. Topol.} \textbf{12} (2012), 1457--1486.

\bibitem{con}
D.~Constantine, Marked length spectrum rigidity in non-positive curvature with singularities. \textit{Indiana Univ. Math. J.} \textbf{67} (2018), no. 6, 2337--2361.

\bibitem{conlaf}
D.~Constantine, J.~F.~Lafont, Marked length rigidity for Fuchsian buildings. \textit{Ergodic Theory Dynam. Systems} \textbf{39} (2019), no. 12, 3262--3291.

\bibitem{coornaert}
 M.~Coornaert, Mesures de Patterson-Sullivan sur le bord d'un espace hyperbolique au sens de Gromov. \textit{Pacific J. Math.} \textbf{159} (1993), 241--270.

\bibitem{CDST} 
R.~Coulon, R.~Dougall, B.~Schapira and S.~Tapie, Twisted Patterson-Sullivan measures and applications to amenability and coverings, to appear in Memoirs of the AMS, arXiv:1809.10881, 2020.

\bibitem{croke} 
C.~B.~Croke, Rigidity for surfaces of nonpositive curvature. \textit{Comment. Math. Helv.} \textbf{65} (1990), no. 1, 150--169.


\bibitem{croke.dairbekov}
C.~B~Croke, N.~Dairbekov, Lengths and volumes in Riemannian manifolds. \textit{Duke Math. J.} \textbf{125} (2004), 1--14.

\bibitem{dfw}
F.~Dahmani, D.~Futer and D.~T.~Wise, Growth of quasiconvex subgroups. \textit{Math. Proc. Cambridge Philos. Soc.} \textbf{167}  (2019), no. 3, 505--530.

\bibitem{dalbo} 
F.~Dal’bo, Remarques sur le spectre des longueurs d’une surface et comptages. \textit{Bol. Soc. Brasil. Mat. (N.S.)} \textbf{30} (1999), no. 2, 199--221.

\bibitem{dlr}
M.~Duchin, C.~J.~Leininger and K.~Rafi, Length spectra and degeneration of flat metrics. \textit{Invent. Math.} \textbf{182} (2010), no. 2, 231--277.

\bibitem{gogolev.hertz} 
A.~Gogolev and F.~Rodriguez Hertz, Abelian Livshits theorems and geometric applications, to appear in Anatole Katok CUP memorial volume, arXiv:2004.14431, 2020.


\bibitem{guilef}
C.~Guillarmou, T.~Lefeuvre, The marked length spectrum of Anosov manifolds. \textit{Ann. Math.} (2) \textbf{190} (1) (2019), 321--344.

\bibitem{ham}
U.~Hamenst\"adt. Cocycles, symplectic structures and intersection. \url{https://arxiv.org/abs/dg-ga/9710009}, arXiv preprint, 1997.

\bibitem{hao}
Y.~Hao. Marked length pattern rigidity for arithmetic manifolds. \url{https://arxiv.org/abs/2206.01336v1}, arXiv preprint, 2022.

\bibitem{kap} 
I.~Kapovich, Random length-spectrum rigidity for free groups. \textit{Proc. Amer. Math. Soc.} \textbf{140} (2012) no. 5, 1549--1560.



\bibitem{MYJ}
K.~Matsuzaki, Y.~Yabuki and J. Jaerisch, Normalizer, divergence type, and Patterson measure for discrete groups of the Gromov hyperbolic space. \textit{Groups Geom. Dyn.} \textbf{14} (2020), no. 2, 369--411.


\bibitem{or}
E.~Oreg\'on-Reyes, Properties of sets of isometries of Gromov hyperbolic spaces. \textit{Groups Geom. Dyn.} \textbf{12} (2018), no. 3, 889--910.

\bibitem{reyes} 
E.~Oreg\'on-Reyes, The space of metric structures on hyperbolic groups. \textit{J. Lon. Math. Soc.} \textbf{107} (2023), 914--942.

\bibitem{otal} 
J.-P.~Otal, Le spectre marqu\'e des longueurs des surfaces \`a courbure n\'egative. (French) \textit{Ann. of Math.} (2) \textbf{131} (1990), no. 1, 151--162.


\bibitem{pollicott.sharp} 
M.~Pollicott, R.~Sharp, Rates of recurrence for $\mathbb{Z}^q$ and $\mathbb{R}^q$ extensions of subshifts of finite type. \textit{J. Lond. Math. Soc.} \textbf{49} (1994), 401--416.

\bibitem{ray} 
B.~Ray, Non-Rigidity of Cyclic Automorphic Orbits in Free Groups. \textit{Internat. J. Algebra Comput.} \textbf{22} (2023), no. 3.

\bibitem{sawyer} 
N.~Sawyer, Partial Marked Length Spectrum Rigidity of Negatively Curved Surfaces, Wesleyan University, 2020.

\bibitem{schwartz.sharp} 
R.~Schwartz, R.~Sharp, The correlation of length spectra of two hyperbolic surfaces. \textit{Comm. Math. Phys.} \textbf{154} (1993), 423--430.

\bibitem{sharp.manhat} 
R.~Sharp, The Manhattan curve and the correlation of length spectra on hyperbolic surfaces. \textit{Math. Z.} \textbf{228} (1998), 745--750.

\bibitem{smillie.vogtmann} 
J.~Smillie, K.~Vogtmann, Length functions and outer space, \textit{Michigan Math. J.} \textbf{39} (1992) no. 3, 485--493.

\bibitem{wu} 
Y.~Wu. Marked Length Spectrum Rigidity for Surface Amalgams. \url{https://arxiv.org/pdf/2310.09968}, arXiv preprint, 2023.


\end{thebibliography}
\end{document}